\documentclass[a4paper]{amsart}

\usepackage{cmbright}
\usepackage{mathrsfs}
\usepackage{xcolor}
\usepackage[utf8]{inputenc}
\usepackage[T1]{fontenc}
\usepackage[english]{babel}
\usepackage{amsfonts}
\usepackage{amssymb}
\usepackage{mathcomp}
\usepackage{mathrsfs}
\usepackage[all]{xy}
\usepackage[bb=boondox]{mathalfa}
\usepackage{graphicx}
\usepackage{epstopdf}
\usepackage[centertags]{amsmath}
\usepackage[pdftex]{hyperref}
\usepackage{vmargin}
\usepackage{tikz-cd}
\usepackage{cases}
\newtheorem{theorem}{Theorem}
\newtheorem{lemma}{Lemma}
\newtheorem{proposition}{Proposition}
\newtheorem{corollary}{Corollary}

\newtheorem{definition}{Definition}
\newtheorem{notation}{\sc Notations}
\newtheorem{remark}{\sc Remark}
\newtheorem{example}{\sc Example}

\newcommand{\mbk}{\mathbb{K}}
\newcommand{\mbn}{\mathbb{N}}

\newcommand{\id}{\mathrm{Id}}

\newcommand{\ov}{\overline}

\newcommand{\Aa}{\mathcal{A}}
\newcommand{\BB}{\mathcal{B}}
\newcommand{\CC}{\mathcal{C}}
\newcommand{\DD}{\mathcal{D}}

\newcommand{\CCC}{\mathscr{C}}
\newcommand{\DDD}{\mathscr{D}}

\newcommand{\QQQ}{\mathscr{Q}}

\newcommand{\ra}{\rightarrow}
\newcommand{\II}{\mathbb{1}}
\newcommand{\III}{\mathbb{I}}

\newcommand{\uCog}{\mathsf{uCog}}

\newcommand{\NilCog}{\mathsf{NilCog}}

\newcommand{\Alg}{\mathsf{Alg}}

\newcommand{\Tfree}{\mathbb{T}}

\newcommand{\Set}{\mathsf{Set}}
\newcommand{\sSet}{\mathsf{sSet}}
\newcommand{\cSet}{\square\mathsf{-Set}}
\newcommand{\ccSet}{\square_c\mathsf{-Set}}
\newcommand{\crSet}{\square_p\mathsf{-Set}}
\newcommand{\ASet}{\mathsf{A-Set}}

\newcommand{\dgMod}{\mathsf{dgMod}}

\newcommand{\catA}{\mathsf{A}}

\newcommand{\catC}{\mathsf{C}}
\newcommand{\catD}{\mathsf{D}}
\newcommand{\catE}{\mathsf{E}}
\newcommand{\catF}{\mathsf{F}}
\newcommand{\catM}{\mathsf{M}}

\newcommand{\catoE}{\mathsf{Cat}_{\mathsf{E}}}
\newcommand{\catoF}{\mathsf{Cat}_{\mathsf{F}}}

\newcommand{\cato}{\mathsf{Cat}}
\newcommand{\catosSet}{\mathsf{Cat}_{\Delta}}

\newcommand{\catoccSet}{\mathsf{Cat}_{\square_c}}
\newcommand{\catocrSet}{\mathsf{Cat}_{\square_p}}

\newcommand{\colim}{\mathrm{colim}}

\newcommand{\poubelle}[1]{}

\title{Cubical model categories and quasi-categories}
\author{Brice Le Grignou}
\email{bricelegrignou@gmail.com}

\date{\today}

\keywords{Enriched categories, cubical sets}

\thanks{The author was supported by the NWO.}

\begin{document}

\maketitle

\begin{abstract}
The goal of this article is to emphasize the role of cubical sets in enriched category theory and infinity-category theory. We show in particular that categories enriched in cubical sets provide a convenient way to describe many infinity-categories appearing in the context of homological algebra.
\end{abstract}

\setcounter{tocdepth}{1}
\tableofcontents

\section*{Introduction}

The goal of this article is to emphasize the role of cubes and cubical sets when dealing with compositions of homotopies.\\

Indeed, let $(\mathsf E, \otimes, \mathbb 1)$ be a monoidal model category, together with the choice of an interval $\mathbb 1 \sqcup \mathbb 1 \hookrightarrow H \xrightarrow{\sim} \mathbb 1$. One can think of the category of simplicial sets with the interval $\Delta[1]$, of the category of chain complexes with the cellular model of the interval, or of the category of differential graded coalgebras with the cellular model of the interval. Then let $(\mathcal A, \gamma , \eta)$ be a monoid in $\mathsf E$ (for instance a simplicial monoid or a differential graded algebra depending on our choice of category $\mathsf E$). A point of $\mathcal A$ is a morphism $a: \mathbb 1 \to \mathcal A$ in the category $\catE$. Then, a path between two points in $\Aa$ is the data of a morphism $H \to \Aa$.
Using the product on $\Aa$, one can define the product of two paths $f$ and $g$ as follows
\[
	H \otimes H \xrightarrow{f \otimes g} \Aa \otimes \Aa \xrightarrow{\gamma} \Aa\ .
\]
The product of $f$ with $g$ is thus a morphism from $H \otimes H$ to $\Aa$ ; that is a square of $\Aa$. Similarly, the product of 
a $n$-cube $f : H^{\otimes n} \to \Aa$ with a $m$-cube $g: H^{\otimes m} \to \Aa$ is a $n+m$-cube $f\cdot g : H^{\otimes n+m} \to \Aa$.
The same phenomenon appears when dealing with a "monoid with many objects", that is a category enriched over $\catE$.
This enlightens the fact that cubes appear naturally when mixing homotopy with composition. In the case where the monoidal structure is Cartesian, that is $\otimes$ is the categorical product $\times$, for instance for simplicial sets, then the interval $H$ has a diagonal map $H \to H \otimes H$. Then, the product of two paths which is a square
induces another path, the diagonal of this square
\[
	H \to H \otimes H \to \Aa \otimes \Aa \xrightarrow{\gamma} \Aa\ .
\]
\newline

There exists a category of cubes $\square_p$ (the $p$ stands for pre), described in details in the book \cite{Cisinski06}, which roughly
consists of cubes of various dimensions $\square_p^n$, together with face inclusions
$\delta_i: \square_p^n \to \square_p^{n+1}$ and contraction along a  direction
$\sigma_i : \square_p^n \to \square_p^{n-1}$. This category has a monoidal product given by  
\[
	\square_p^n \otimes \square_p^m = \square_p^{n + m}\ .
\]
In a similar way as simplicial sets $\sSet = \mathsf{Fun}(\Delta^{\mathrm{op}}, \Set)$ are gluings of points, lines, triangles, tetrahedrons, \ldots,
precubical sets $\crSet = \mathsf{Fun}(\square_c^{\mathrm{op}}, \Set)$ are gluings of  points, lines, squares, cubes, \ldots Moreover, precubical sets represent all the possible homotopy types as well as simplicial sets ; indeed, the category $\crSet$ has a model structure Quillen equivalent to the Kan--Quillen model category of simplicial sets. Besides, cubical sets inherit a monoidal structure from that of cubes.\\

Then all the discussion above about compositions and homotopies is encompassed in the following proposition.

\begin{proposition}[\cite{Cisinski06}]
 Let $(\mathsf E, \otimes, \mathbb 1)$be a monoidal model category. Then, the data of an interval $H$ of $\catE$, is essentially the data of a monoidal Quillen adjunction
 \[
	\begin{tikzcd}[ampersand replacement=\&]
		\crSet \arrow[r,shift left,"L_p^H"] \& \catE\ . \arrow[l,shift left,"R_p^H"]
	\end{tikzcd}
\]
Moreover, this induces another adjunction
 \[
	\begin{tikzcd}[ampersand replacement=\&]
		\catocrSet \arrow[r,shift left,"L_p^H"] \& \catoE \arrow[l,shift left,"R_p^H"]
	\end{tikzcd}
\]
between categories enriched in precubical sets and categories enriched in $\catE$, which is a Quillen adjunction, when the Dwyer-Kan model structure on $\catoE$ exists.
\end{proposition}

Besides, let us consider the simplicial set $\Delta[3]$.
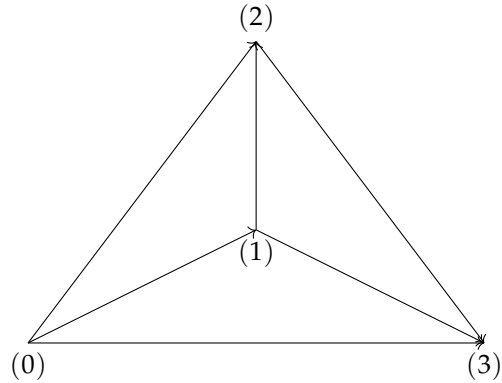
\[
\begin{tikzpicture}
    \draw [->] (0,0) -- (6,0) ;
    \draw [->] (0,0) -- (3,4) ;
    \draw [->] (0,0) -- (3,1.5) ;
    \draw [->] (3,1.5) -- (6,0) ;
    \draw [->] (3,1.5) -- (3,4) ;
    \draw [->] (3,4) -- (6,0) ;
    \draw (0,0) node[below] {$(0)$} ;
    \draw (6,0) node[below] {$(3)$} ;
    \draw (3,4) node[above] {$(2)$} ;
    \draw (3,1.5) node[below] {$(1)$} ;
\end{tikzpicture}
\]
Seen as an infinity-category, it has $4$ objects, that is
$0,1,2,3$. Its morphisms are generated by the edges $(ij)$ for any integers $0 \leq i<j \leq 3$. Then there are exactly $4$ morphisms from $0$ to $3$ that is
$(03)$, $(01) (13)$, $(02) (23)$ and $(01)  (12)(23)$. They are organized into a square a follows.
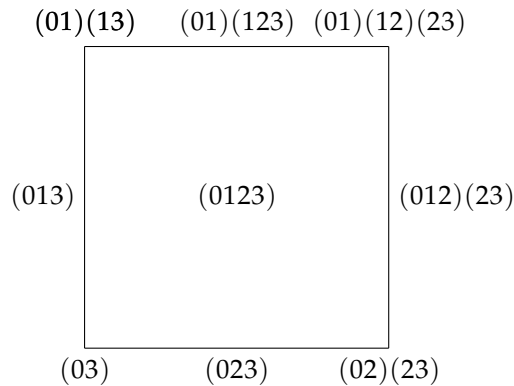
\[
\begin{tikzpicture}
    \draw (0,0) -- (4,0) ;
    \draw (0,0) -- (0,4) ;
    \draw (4,0) -- (4,4) ;
    \draw (0,4) -- (4,4) ;
    \draw (0,0) node[below] {$(03)$} ;
    \draw (0,4) node[above] {$(01)(13)$} ;
    \draw (4,0) node[below] {$(02)(23)$} ;
    \draw (4,4) node[above] {$(01)(12)(23)$} ;
    \draw (0,4) node[above] {$(01)(13)$} ;
    \draw (0,2) node[left] {$(013)$} ;
    \draw (2,0) node[below] {$(023)$} ;
    \draw (4,2) node[right] {$(012)(23)$} ;
    \draw (2,4) node[above] {$(01)(123)$} ;
    \draw (2,2) node {$(0123)$} ;
\end{tikzpicture}
\]
More generally, the morphisms of $\Delta[n]$ from $i$ to $j$ are organized into a $j-i-1$-cube for $i<j$. In particular the morphisms from $0$ to $n$ are organized into a $n-1$-cube. The face maps $\Delta[n-1] \to \Delta[n]$ induce face maps between cubes. This seems to be the beginning of a functor $W_p$ from $\Delta$ to the category $\catocrSet$ of categories enriched over precubical sets such that $W_{p,n} := W_p(n)$ would be the cubical category with $n+1$ objects $0,\ldots, n$ and 
\[
	W_{p,n} (0,n) = \square_p [n-1] \ ,
\]
and more generally,
\[
	W_{p,n} (i,j) = \square_p [j-i-1] \ ,
\]
for any integers $0 \leq i < j \leq n$. However, the degeneracy map $\sigma_1: \Delta[3] \to \Delta[2]$ would give a functor $W_{p,2} \to W_{p,1}$ corresponding at the level of mapping spaces to a map $\gamma: \square_p[2] \to \square_p[1]$ mimicking the behavior of the function
\begin{align*}
 [0,1]\times[0,1] &\to [0,1]\\
 (x,y) &\mapsto \mathrm{max}(x,y)\ .
\end{align*}
Unfortunately, such a morphism from $\square_p[2]$ to $\square_p[1]$ does not exist. Therefore, one needs to enhance precubes and precubical sets by adding this map $\gamma$ to obtain respectively the category $\square_c$ of cubes with connections and the category $\ccSet$ of cubical sets with connections ; see for instance \cite{Maltsiniotis09}.\\

Actually, there exists a functor $W_c$ from the category  $\Delta$ to the category $\catoccSet$ of categories enriched over cubical sets with connections which has the shape that we hoped earlier. Indeed, $W_n := W_c (n)$ has $n+1$ objects $0,\ldots, n$ and 
\[
	W_{n} (0,n) = \square_c [n-1] \ .
\]
It induces an adjunction
 \[
	\begin{tikzcd}[ampersand replacement=\&]
		\sSet \arrow[r,shift left,"W_c"] \& \catoccSet\ . \arrow[l,shift left,"N^c"]
	\end{tikzcd}
\]
The usual adjunctions relating simplicial sets to categories enriched over a monoidal model category $\catE$ factorizes through this one. Moreover, this is a Quillen adjunction if the category of simplicial sets is endowed with the Joyal model structure. This two facts coupled with some Reedy theory lead us to the following theorem.

\begin{theorem}
 Let $\catE$ be a monoidal model category and suppose that the category $\catoE$ of categories enriched over $\catE$ has a Dwyer-Kan model structure. Then, for any Reedy cofibrant replacement $F$ of the cosimplicial $\catE$-enriched category $n \mapsto [n]$, the induced adjunction
  \[
	\begin{tikzcd}[ampersand replacement=\&]
		\sSet \arrow[r,shift left,"F_!"] \& \catoE\ , \arrow[l,shift left,"F^!"]
	\end{tikzcd}
\]
is a Quillen adjunction.
\end{theorem} 

The use of cubical sets is particularly efficient when dealing with enriched model structures. Simplicial model categories are model categories $\catM$ enriched, tensored and cotensored over the category of simplicial set satisfying an additional axiom which implies that the simplicial category of fibrant-cofibrant objects is a model of the infinity-category that is presented by the model category $\catM$. Replacing the category of simplicial sets by another monoidal model category $\catE$, one obtains the notion of an $\catE$-model category ; see \cite{Hovey99}. By the following proposition, any $\catE$-model category $\catM$ has the structure of a cubical model category.

\begin{proposition}
  Let $(\mathsf E, \otimes, \mathbb 1)$ be a monoidal model category and let $\catM$ be an $\catE$-model category. Then, any choice of an interval (resp. monoidal interval) $H$ in $\catE$ induces the structure of a $\crSet$-model category (resp. $\ccSet$-model category) on $\catM$.
\end{proposition}

In particular, if $\catM$ is a simplicial model category, then it has an induced structure of a cubical model category. We are also interested in the cases where $\catM$ is an $\catE$-model category for some $\catE$ but it is not a simplicial model category ; for instance the Joyal model category and the model category
of algebras over a nonsymmetric differential graded operad where $\catE$ is respectively the Joyal model category and the model category of differential graded coalgebras.

\subsection*{Layout}

This article is organized as follows. In the first section, we describe in details the theory of enriched categories and recall some results about their homotopy theory. The two next sections deal with the categories of precubical sets and cubical sets with connections and their model structures. Many of the results  given there are due to Cisinski. The fourth part makes a link between simplicial sets and enriched categories. The final section applies the material developed to concrete examples.

\subsection*{Relations to other works}

Some results given here were already known. Indeed, many of the results about cubical homotopy are consequences of the work of Cisinski (\cite{Cisinski06}). The idea that the homotopy coherent nerve functors for dg categories and simplicial categories factor through categories enriched over cubical sets with connections was already in \cite{RiveraZenalian16}. Moreover, similar ideas to those of Section 4, already appeared independently in \cite{KapulkinVoevodsky18}. However, to the best of my knowledge, this is the first time that cubical sets are used systematically to study enriched categories. 

\subsection*{Conventions}

\begin{itemize}
\item[$\triangleright$] The category of simplicial sets is denoted $\sSet$. It is usually endowed with the Kan-Quillen model structure. If we endow it with the Joyal model structure, we write $\sSet_J$.
\item[$\triangleright$] We denote by $[n]$ the category with $n+1$ objects $0,\ldots n$ such that
\[
	\hom_{[n]}(i,j) = 	
\begin{cases}
 * \text{ if }i \leq j\ ,\\
 \emptyset \text{ otherwise.}
\end{cases}
\]
Furthermore, $\delta_i^\Delta: [n]\to [n+1]$ is the only injective functor which omits the objects $i$ in $[n+1]$ and $\sigma_i^\Delta: [n]\to [n-1]$ is the only surjective functor which sends the objects $i$ and $i+1$ to $i$. All these categories $[n]$ and these functors generate the category $\Delta$.
 \item[$\triangleright$] Let $\mathcal U < \mathcal U'$ be two universes. Usually, we work with categories whose sets of morphisms are $\mathcal U$-small and whose set of objects is $\mathcal U$-large (that is a subset of $\mathcal U$). In particular, these categories are $\mathcal{U}'$-small, in the sense that their sets of objects and morphisms are $\mathcal{U}'$-small. Then, when performing constructions on a category considered as an object, or when working with <<the category of categories>>, we assume working with $\mathcal{U}'$-small categories.
\end{itemize}

\subsection*{Acknowledgment}
I was supported by the NWO Spinoza grant of Pr. Ieke Moerdijk. Moreover, I would like to thank Gabriel Drummond-Cole and Manuel Rivera for pointing out to me the existence of the article \cite{RiveraZenalian16}. I would also like to thank the anonymous referee for his remarks.


\section{Monoidal categories and enriched categories}

\subsection{Monoidal categories and monoidal adjunctions}

In this section, we recall the notions of a monoidal adjunction, of a monoidal natural transformation after Kelly \cite{Kelly74}.

\begin{definition}
 Let  $(\catE, \otimes , \mathbb 1)$ and  $(\catF, \otimes , \mathbb 1)$ be two monoidal categories, and let $F,G:\catC \to \catD$ be two lax monoidal functors between them. A natural transformation $\phi: F\to G$ is monoidal if the following diagrams commute  
 \[
	\begin{tikzcd}[ampersand replacement=\&]
		{F(X)\otimes F(Y)} \arrow[r] \arrow[d,"\phi (X)\otimes \phi (Y)"'] \& {F(X\otimes Y)} \arrow[d,"\phi(X\otimes Y)"]\\
		{G(X)\otimes G(Y)} \arrow[r] \& G(X\otimes Y)
	\end{tikzcd}
	\ \ \ \ 
	\begin{tikzcd}[ampersand replacement=\&]
		\&\II \arrow[ld] \arrow[rd]\\
		F(\II) \arrow[rr,"\phi(\II)"']\&\& G(\II)\ ,
	\end{tikzcd}
\]
for any objects $X,Y$ of the category $\catC$.
\end{definition}

\begin{definition}
 Let  $(\catE, \otimes , \mathbb 1)$ and  $(\catF, \otimes , \mathbb 1)$ be two monoidal categories. A monoidal adjunction between them is the data of an adjunction 
\[
	\begin{tikzcd}[ampersand replacement=\&]
		\catE \arrow[r,shift left,"L"] \& \catF\ , \arrow[l,shift left,"R"]
	\end{tikzcd}
\]
together with structures of lax monoidal functors on $L$ and $R$ so that the unit map $\id{} \to RL$ and the counit map $LR \to \id{}$
are monoidal natural transformations.
\end{definition}

\begin{theorem}\cite{Kelly74}
\label{lemma:unitmonoidal}
\label{lemmamonoidalajunction}
 Let us consider an adjunction $L \dashv R$ between two monoidal categories.
 
\begin{enumerate}
 \item the data of a structure of a lax monoidal functor on $R$ is equivalent to the data of a
  structure of an oplax monoidal functor on $L$;
 \item given an enhancement of $L \dashv R$ into a monoidal adjunction, the structural maps
 $L(X \otimes Y) \to L(X) \otimes L(Y)$ and $L(\II) \to \II$ making $L$ an oplax monoidal functor (as a left adjoint to the
 lax monoidal functor $R$) are inverse to the structural maps $L(X) \otimes L(Y) \to L(X \otimes Y)$ and $\II \to L(\II)$
 making $L$ a lax monoidal functor;
 \item a structure of a lax monoidal functor on $R$ is part of a monoidal adjunction if and only if the corresponding
 structure of an oplax monoidal functor on $L$ is strong.
\end{enumerate}
\end{theorem}

Given a structure of an oplax monoidal structure on $L$, the 
 structure of a lax monoidal functor on $R$ is given by the adjoint morphisms of the maps
\begin{align*}
 	&L(R(X) \otimes R(Y)) \to LR(X) \otimes LR(Y) \to X \otimes Y
	\\
	& L(\II) \to \II .
\end{align*}
 Similarly, given a structure of a lax monoidal structure on $R$, the 
 structure of an oplax monoidal functor on $L$ is given by the adjoint morphisms of the maps
\begin{align*}
 	&X \otimes Y \to RL(X) \otimes RL(Y)\to R(L(X) \otimes L(Y))
	\\
	& \II \to R(\II) .
\end{align*}

\begin{remark}
 Following Theorem \ref{lemmamonoidalajunction}, one can also define a monoidal adjunction
 as an adjunction together with the structure of a strong monoidal functor on the left adjoint.
\end{remark}

\begin{definition}
 A bilinear monoidal category $(\catE, \otimes , \mathbb 1)$ is a monoidal category such that $\catE$ is cocomplete and such that the bifunctor $-\otimes -$ commutes with colimits separately on both sides.
\end{definition}

For any such bilinear monoidal category, the element $\II \in \catC$ induces a cocontinuous functor $i: \Set \to \catE$ such that $i(*) = \II$.This functor has a right adjoint $S$ such that $S(X) = \hom_{\catE}(\II, X)$.

 \[
	\begin{tikzcd}[ampersand replacement=\&]
		\Set \arrow[r,shift left,"i"] \& \catE \arrow[l,shift left,"S"]
	\end{tikzcd}
\]

\begin{proposition}
 The functor $i$ is strong monoidal. Therefore, the adjunction $i \dashv S$ is monoidal.
\end{proposition}

\begin{proof}
Since the monoidal structure is bilinear, then for any sets $X$ and $Y$ we have
\[
i(X) \otimes i(Y) \simeq  \left(\sqcup_{a \in X } \II \right) \otimes \left(\sqcup_{b \in Y } \II \right) \simeq \sqcup_{(a,b) \in X \times Y} \II \otimes \II \simeq \sqcup_{(a,b) \in X \times Y} \II  \simeq i(X \times Y)\ .
\]
\end{proof}

\subsection{Day convolution product}

Let $(\catA,\otimes,\II)$ be a small category endowed with a monoidal structure. Then, the opposite category $\catA^{op}$ inherits a monoidal structure. We denote by $\ASet$ the category of presheaves over $\catA$, that is functors from $\catA^{op}$ to $\Set$.

\begin{definition}
For any presheaves $X,Y$ over $\catA$, the Day convolution product $X \otimes Y$ is the following left Kan extension
 \[
	\begin{tikzcd}[ampersand replacement=\&]
		\catA^{op} \times \catA^{op}
		\arrow[r, "X \times Y"]
		\arrow[d, "\otimes", swap]
		\& \Set \times \Set \arrow[r, "\times"]  
		\& \Set\\
		\catA^{op}\ . \arrow[rru, dotted,"X \otimes Y"']
	\end{tikzcd}
\]

\end{definition}

The Day product may also be defined in the following way. Both $X$ and $Y$ are colimits of representables
\[
\begin{cases}
	X \simeq \colim_{a \in A/X} \ a\ ,\\
	Y \simeq \colim_{a \in A/Y} \ a\ .
\end{cases}
\]
Then,
\[
X\otimes Y \simeq \colim_{(a,a') \in A/X \times A/Y} \ a \otimes a'\ .
\]

\begin{proposition}\cite{Day70}
 The Day product defines a bilinear monoidal structure on the category $\ASet$. Moreover, the Yoneda embedding functor $\catA\to \ASet$ is strong monoidal.
\end{proposition}

\subsection{Enriched categories}

\begin{definition}
 Let  $(\catE, \otimes , \mathbb 1)$ be a monoidal category. A category enriched over $\catE$ (or $\catE$-category) $(\CCC,m,u)$ is the data of
\begin{itemize}
 \item[$\triangleright$] a set of objects $Ob(\CCC)$,
 \item[$\triangleright$] for any objects $x,y\in Ob(\CCC)$, an element $\CCC(x,y)$ of the category $\catE$,
 \item[$\triangleright$] an associative composition $m_{x,y,z}: \CCC(x,y) \otimes \CCC(y,z) \to \CCC(x,z)$,
 \item[$\triangleright$] a unit for this composition $u_x: \II \to \CCC(x,x)$ for any object $x$.
\end{itemize}
A functor $F$ between two such $\catE$-categories $(\CCC,m,u)$ and $(\CCC',m',u')$ is the data of
\begin{itemize}
 \item[$\triangleright$] a function from $Ob(\CCC)$ to $Ob(\CCC')$ also denoted $F$,
 \item[$\triangleright$] for any objects $x,y\in Ob(\CCC)$, a morphism $F_{x,y}: \CCC(x,y) \to \CCC'(F(x),F(y))$,
 \item[$\triangleright$] which commutes with the composition and the unit in the sense that
\[
\begin{cases}
 m'_{F(x),F(y),F(z)} \left( F_{x,y} \otimes F_{y,z} \right) &= F_{x,z} m_{x,y,z}\ ,\\
 u_{F(x)} =F_{x,x}u_x\ ,
\end{cases}
\]
for any objects $x,y,z$ of $\CCC$.
\end{itemize}
This defines the category $\catoE$ of $\catE$-categories.
\end{definition}

Forgetting the composition and the unit in $\catE$-categories, one gets the notion of a $\catE$-quiver.

\begin{definition}
 A $\catE$-quiver $\QQQ$ is the data of a set of objects $Ob(\QQQ)$ together with an element $\QQQ(x,y)$ of $\catE$ for any objects $x,y\in Ob(\QQQ)$. A morphism $F$ of quivers from $\QQQ$ to $\QQQ'$ is the data of a function from $Ob(\QQQ)$ to $Ob(\QQQ')$ also denoted $F$ and, for any objects $x,y\in Ob(\QQQ)$, a morphism $F_{x,y}: \QQQ(x,y) \to \QQQ'(F(x),F(y))$. We denote by $\mathsf{Quiv}_{\catE}$ the category of $\catE$-quivers.
\end{definition}

\begin{lemma}
Suppose that the monoidal category $\catE$ is bilinear. Then, the forgetful functor $O: \catoE \to \mathsf{Quiv}_{\catE}$ has a left adjoint $\Tfree$ such that for any quiver $\QQQ$
\begin{itemize}
 \item[$\triangleright$] the set $Ob(\Tfree\QQQ)$ is exactly the set $Ob(\QQQ)$,
 \item[$\triangleright$] for any objects $x,y \in Ob(\QQQ)$
 \[
\begin{cases}
 \Tfree\QQQ(x,y) = \sqcup_{n\geq 1} \sqcup_{x_0=x, x_1, \ldots, x_n=y} \QQQ(x_0,x_1) \otimes \QQQ(x_1,x_2)\otimes \cdots \otimes \QQQ(x_{n-1},x_n)\text{ if }x\neq y\ ,\\
  \Tfree\QQQ(x,x) =  \II \sqcup \left(\sqcup_{n\geq 1} \sqcup_{x_0=x, x_1, \ldots, x_n=x} \QQQ(x_0,x_1) \otimes \QQQ(x_1,x_2)\otimes \cdots \otimes \QQQ(x_{n-1},x_n)\right)\ ,
\end{cases}
 \]
 \item[$\triangleright$] the composition is given by the concatenation of tensors.
\end{itemize}
 Moreover, the adjunction $\Tfree \dashv O$ is monadic.
\end{lemma}

\begin{proof}
Straightforward.
\end{proof}

\begin{lemma}
If the category $\catE$ is cocomplete, then the category  $\mathsf{Quiv}_{\catE}$ is cocomplete. Moreover, for any regular cardinal $\lambda$, if $\catE$ is a $\lambda$-presentable category, then, the category $\mathsf{Quiv}_{\catE}$ is $\lambda$-presentable.
\end{lemma}

\begin{proof}
 It is straightforward to prove that the category $\mathsf{Quiv}_{\catE}$ is stable under small coproducts. Let us prove that it has all cokernels. Consider the following diagram of $\catE$-quivers.
 \[
	\begin{tikzcd}[ampersand replacement=\&]
		\QQQ \arrow[r,shift left,bend left=25,"F"] \arrow[r,shift right,bend right=25,"G"'] \& \QQQ' 
	\end{tikzcd}
\]
The cokernel $\QQQ''$ of $F$ and $G$ is the following $\catE$-quiver:
\begin{itemize}
 \item[$\triangleright$]  its set of objects is the cokernel of the underlying functions of $F$ and $G$ from $Ob(\QQQ)$ to $Ob(\QQQ')$. Therefore, it is the quotient of the set $Ob(\QQQ'')$ by the relation $F(x)\sim G(x)$ for any object $x$ of $\QQQ$. Let us denote by $K$ the surjection from $Ob(\QQQ')$ to $Ob(\QQQ'')$. It is clear that at the level of objects of $\QQQ$, $KF=KG$.
 \item[$\triangleright$] For any objects $x,y \in Ob(\QQQ'')$, $\QQQ''(x,y)$ is the cokernel in $\catE$ of the following diagram.
  \[
	\begin{tikzcd}[ampersand replacement=\&]
		\sqcup_{KF(a)=x,KF(b)=y}\QQQ(a,b) \arrow[r,shift left,bend left=25,"F_{a,b}"] \arrow[r,shift right,bend right=25,"G_{a,b}"'] \& \sqcup_{K(x')=x,K(y')=y}\QQQ'(x',y') 
	\end{tikzcd}
\]
\end{itemize}
Besides, if $\catE$ is $\lambda$-presentable, then the category $\mathsf{Quiv}_{\catE}$ of $\catE$-quivers is generated under $\lambda$-filtered colimits by $\catE$-quivers $\QQQ$ whose sets of objects are $\lambda$-small and such that for any objects $x,y$, $\QQQ(x,y)$ is $\lambda$-small. The (possibly large) set of isomorphisms classes of such $\catE$-quivers is actually a small set.
\end{proof}

\begin{lemma}
Suppose that the monoidal category $(\catE,\otimes,\II)$ is bilinear. Then the category $\catoE$ has all filtered colimits and the forgetful functor $O$ preserves filtered colimits.
\end{lemma}

\begin{proof}
Let $D: I \to \catoE$ be a filtered diagram and let 
\[
	\QQQ = \colim_I\  O \circ D\ .
\]
We denote by $F(i)$ the morphism of $\catE$-quivers $D(i) \to \QQQ$ for any object $i \in I$.
The set of objects $Ob(\QQQ)$ is the colimit of the diagram $i \in I \mapsto Ob(D(i))$. 
Moreover, for any two of its element $x,y$,
\[
	\QQQ(x,y) = \colim_{(i,F(i)(x')=x,F(i) (y')=y) } D(i) (x',y') \simeq \colim_i (\sqcup_{(F(i)(x')=x,F(i) (y')=y)}  D(i) (x',y'))\ .
\]
Then, for any three objects $x,y,z$ of $\QQQ$,
\begin{align*}
 	\QQQ(x,y) \otimes \QQQ(y,z) &=\left( \colim_{(i,F(i)(x')=x,F(i) (y')=y) } D(i) (x',y')\right)
	\otimes \left( \colim_{(j,F(j)(y'')=y,F(j) (z'')=z) } D(j) (y'',z'') \right)\\
	&\simeq \colim_{(i,F(i)(x')=x,F(i) (y')=y, j,F(j)(y'')=y,F(j) (z'')=z)} D(i) (x',y') \otimes D(j)(y'',z'')\ .
\end{align*}
The inclusion functor $(i,F(i)(x')=x,F(i) (y')=y, F(i)(z')=z ) \mapsto (i,F(i)(x')=x,F(i) (y')=y, i,F(i)(y')=y,F(i) (z')=z)$ is final. Thus 
\[
	\QQQ(x,y) \otimes \QQQ(y,z) \simeq \colim_{(i,F(i)(x')=x,F(i) (y')=y, F(i)(z')=z )} D(i) (x',y') \otimes D(i)(y',z')\ .
\]
Besides, the following cocone
\[
	D(i) (x',y') \otimes D(i)(y',z') \xrightarrow{m_{x',y',z'}} D(i)(x',z') \xrightarrow{(F(i))_{x',y'}} \QQQ(x,y)\ ,
\]
induces a composition morphism $m_{x,y,z}^\QQQ:\QQQ(x,y) \otimes \QQQ(y,z) \to \QQQ(x,z)$. Moreover, the composite map
\[
	\II \xrightarrow{\eta_{x'}} D(i)(x',x') \xrightarrow{F(i)_{x',x'}} \QQQ(x,x)
\]
does not depend on the choice of $i$ and $x' \in Ob(D(i))$ such that $F(i)(x') =x$. Then, it defines a morphism $\eta_x: \II \to \QQQ(x,x)$.
Similar arguments about final diagrams as those used above show that
$\eta_x$ is a unit for the composition and that the composition is associative. We thus have defined the structure of an $\catE$-category on $\QQQ$. Finally, it is straightforward to prove that $\QQQ$ equipped with this structure is the colimit of the diagram $D$.
\end{proof}

\begin{theorem}\cite{KellyLack00}
Let $\lambda$ be regular cardinal. Suppose that the category $\catE$ is a $\lambda$-presentable bilinear monoidal category. Then the category $\catoE$ is 
$\lambda$-presentable.
\end{theorem}

\begin{proof}
It follows from the fact that the monad $O \circ \mathbb T$ preserves filtered colimits.
\end{proof}

\begin{definition}
We denote by $*_{\mathbb 1}$ the $\catE$-enriched category with one object $0$ such that
\[
	*_{\mathbb 1} (0,0) = \mathbb{1}_\catE\ .
\]
\end{definition}

\begin{definition}
 For any object $X$ of $\catE$, let $[1]_X$ be the $\catE$-category with two objects $0$ and $1$ such that
 \[
\begin{cases}
[1]_X(0,0)= [1]_X(1,1)= \II\ ,\\
[1]_X(0,1)= X\ ,\\
[1]_X(1,0)=\emptyset\ .
\end{cases}
 \]
 This defines a functor $[1]: \catE \to \catoE$.
\end{definition}

\subsection{Adjunction between categories of enriched categories}

Let $G$ be a lax monoidal functor from $(\catE, \otimes , \mathbb 1)$ to  $(\catF, \otimes , \mathbb 1)$. Since $G$ is monoidal one can define a functor from $\catoE$ to $\catoF$ also denoted $G$ such that for any $\catE$-category $(\CCC,m,u)$:
\begin{itemize}
 \item[$\triangleright$] $Ob\left(G(\CCC)\right) = Ob(\CCC)$,
 \item[$\triangleright$] $G(\CCC)(x,y) = G(\CCC(x,y))$,
 \item[$\triangleright$] the composition is defined as follows
 \[
 G(\CCC(x,y)) \otimes G(\CCC(y,z)) \rightarrow G(\CCC(x,y)\otimes \CCC(y,z)) \xrightarrow{G(m_{x,y,z})} G(\CCC(x,z))\ .
 \]
\end{itemize}

\begin{proposition}\label{prop:extendadj}
 Let $L\dashv R$ be a monoidal adjunction between two monoidal categories $(\catE, \otimes , \mathbb 1)$ and  $(\catF, \otimes , \mathbb 1)$. Then, the extended functor $L: \catoE \to \catoF$ is left adjoint to the extended functor $R: \catoF \to \catoE$.
\end{proposition}

\begin{proof}
For any $\catE$-category $(\CCC,m,u)$, the map $\eta_\CCC: \CCC \to RL(\CCC)$ is indeed a functor of $\catE$-categories. Similarly, for any $\catF$-category $(\CCC',m',u')$, the map $\epsilon_{\CCC'}:LR(\CCC') \to \CCC'$ is indeed a functor of $\catF$-categories. It is then straightforward to prove that the composite functors
\begin{align*}
 & L(\CCC) \xrightarrow{L(\eta_\CCC)} LRL(\CCC) \xrightarrow{\epsilon_{L(\CCC)}} L(\CCC)\ ,\\
 & R(\CCC') \xrightarrow{\eta_{R(\CCC')}} RLR(\CCC') \xrightarrow{R(\epsilon_{\CCC'})} R(\CCC')\ ,\\
\end{align*}
are respectively the identity of $L(\CCC)$ and the identity of $R(\CCC')$.
\end{proof}

\subsection{Monoidal model categories}

\begin{definition}[Monoidal model category]
Let $(\mathsf E, \otimes , \mathbb 1)$ be a (not necessarily symmetric) monoidal category equipped with a model structure. It is said to be a monoidal model category if
\begin{itemize}
 \item[$\triangleright$] the monoidal structure is bilinear,
 \item[$\triangleright$] the monoidal unit $\mathbb 1$ is cofibrant,
 \item[$\triangleright$] for any cofibrations $f: X \to X'$ and $g:Y \to Y'$, the morphism
 \[
 X' \otimes Y \sqcup_{X \otimes Y} X \otimes Y' \to X' \otimes Y'
 \]
 is a cofibration; moreover, this is a weak equivalence whenever either $f$ or $g$ is.
\end{itemize}
\end{definition}

\begin{remark}
Our definition is different from Hovey's original definition \cite[Definition 4.2.6]{Hovey99}. Indeed, we do not assume that the monoidal structure is closed,
but we assume that the monoidal unit is cofibrant, which is stronger than Hovey’s unit axiom.
\end{remark}

A consequence of K. Brown's lemma (a functor that sends acyclic cofibrations between cofibrant objects to weak equivalences
preserves weak equivalences between cofibrant objects) is that,
in a monoidal model category, for any cofibrant object $X$ and for any weak equivalence $Y \to Y'$ between cofibrant objects,
the morphism $X \otimes Y \to X \otimes Y'$ is a weak equivalence between cofibrant objects as well as the morphism $Y \otimes X \to Y' \otimes X$.
Then for any weak equivalence between cofibrant objects, the morphism $X^{\otimes n} \to Y^{\otimes n}$ is a weak equivalence between cofibrant objects for any integer $n \in \mbn$.

If $(\mathsf E, \otimes , \mathbb 1)$ is a monoidal model category, then we can define a tensor product on the homotopy category $Ho(\catE)$ as follows:
\[
	X \otimes_{Ho(\catE)} Y := \pi (QX \otimes QY)\ ,
\]
where $QX$ and $QY$ are cofibrant replacement of $X$ and $Y$ and $\pi$ is the localisation functor $\pi: \catE \to Ho(\catE)$.

\begin{proposition}\cite[Theorem 4.3.2]{Hovey99}
 This tensor products is part of a monoidal structure on the category $Ho(\catE)$. Moreover, the localisation functor $\pi: \catE \to Ho(\catE)$ is lax monoidal.
\end{proposition}

\begin{remark}
 Again, Hovey's theorem involves closed monoidal structure and imply results about closedness of the monoidal structure of the monoidal structure of
 $Ho(\catE)$. However, the proof of the part of the theorem that we recall does not use the closedness.
\end{remark}

\begin{definition}
 Let $(\catE,\otimes ,\II)$ and $(\catF,\otimes ,\II)$ be two monoidal model categories. A Quillen monoidal adjunction relating $\catE$ to $\catF$ is an adjunction
\[
	\begin{tikzcd}[ampersand replacement=\&]
		\catE \arrow[r,shift left,"L"] \& \catF \arrow[l,shift left,"R"]
	\end{tikzcd}
\]
which is both a Quillen adjunction and a monoidal adjunction.
\end{definition}

\begin{lemma}\label{lemmamonoidalquillenadj}
 Let us use the notations of the above definition. Then the left derived functor $\mathbb L L: Ho(\catE) \to Ho(\catF)$ is strong monoidal. Moreover, the canonical natural transformation from $\mathbb L L \circ \pi_{Ho(\catE)}$ to $\pi_{Ho(\catE)} \circ L$ is a monoidal natural transformation.
\end{lemma}

\begin{proof}
The structure of a strong monoidal functor on the left derived functor is given by the isomorphism
 \[
 	\II_{Ho(\catF)} = \II_{\catF} \simeq L(\II_{\catE}) = \mathbb L L (\II_{Ho(\catE)})
 \]
and the following composite map
\begin{align*}
 	\mathbb L L (X) \otimes_{Ho(\catF)} \mathbb L L (Y) 
	&= L(QX) \otimes_{Ho(\catF)} L(QY) 
	\\
	&\simeq L(QX) \otimes_{\catF} L(QY)
	\\
	&\simeq L(QX \otimes_{\catF} QY)
	\\
	&=  \mathbb L L (X \otimes_{\catE} Y ).
\end{align*}
where the isomorphism $L(QX) \otimes_{Ho(\catF)} L(QY)  \simeq L(QX) \otimes_{\catF} L(QY)$ comes from the fact that both
$L(QX)$ and $L(QY)$ are cofibrant. A straightforward check shows that this defines the structure of a strong monoidal functor and the canonical natural transformation from $\mathbb L L \circ \pi_{Ho(\catE)}$ to $\pi_{Ho(\catE)} \circ L$ is a monoidal natural transformation.
\end{proof}

\subsection{Enriched, tensored an cotensored category}

In this section, we recall the definition of tensored-cotensored-enriched category over a not necessarly symmetric monoidal category.

\begin{definition}
Let $(\catE, \otimes , \mathbb 1)$ be a monoidal category . Any category $\CCC$ enriched over $\catE$
has an underlying category $S(\CCC)$ with the same object and so that
\[
	\hom_{S(\CCC)}(x,y) = \hom_{\catE}(\mathbb 1, \CCC(x,y)) .
\]
Then the mapping $(x,y) \in Ob(\CCC)\times Ob(\CCC) \mapsto \CCC(x,y) \in \catE$
becomes a functor
\[
	\{-,-\} : S(\CCC)^{op} \times S(\CCC) \to \catE . 
\]
\end{definition}

\begin{remark}
Actually, one can define a category enriched over $\catE$ as the data
of a category $\catC$ together with a functor $\{-,-\} : \catC^{op} \times \catC \to \catE$
together with a natural maps $\{X, Y\}\otimes  \{Y,Z\} \to \{X,Z\}$ and maps $\II_\catE \to \{X,X\}$
that make a unital associative composition and so that
\[
	\hom_{\catC}(x,y) = \hom_{\catE}(\mathbb 1, \{x,y\}) .
\]
\end{remark}

\begin{definition}[(Co)tensorisation]\label{defin:tensored} \label{def:tce}
  Let $(\catE, \otimes , \mathbb 1)$ be a monoidal category and let $\catC$ be a category
  enriched over $\catE$.
\begin{itemize}
\item[$\triangleright$] One says that $\catC$ is tensored over $\catE$ if the
functor 
\[
	Y \in \catC \mapsto \{X,Y\} \in \catE
\]
has a left adjoint that we denote $X \triangleleft -$
\item[$\triangleright$] One says that $\catC$ is cotensored over $\catE$ if the
functor
\[
	X \in \catC^{op} \mapsto \{X,Y\} \in \catE
\]
has a right adjoint that we denote $\langle -, Y \rangle$.
\end{itemize}
\end{definition}

If $\catC$ is tensored over $\catE$, then the construction $(X,A) \mapsto X \triangleleft \Aa$
is binatural. Thus, one gets a functor
 $$
 - \triangleleft - : \catC \times \catE \ra \catC
 $$
Moreover, the composition and the unit of the enrichment
induce functorial morphisms
 $$
 \begin{cases}
 X \triangleleft (\Aa \otimes \BB)  \to (X \triangleleft \Aa ) \triangleleft \BB \ ,\\
X \triangleleft \II  \simeq X\ ,
\end{cases}
 $$
for any $X \in \catC$, any $\Aa, \BB \in \catE$; these functors are compatible with the monoidal structure of $\catE$ in the sense that the following diagrams are commutative
 $$
 \xymatrix{X \triangleleft \big( (\Aa \otimes \BB ) \otimes \CC \big) \ar[r] \ar[d] 
 &  \big(X \triangleleft (\Aa \otimes \BB )\big) \triangleleft \CC  \ar[r] 
 & \big( (X \triangleleft \Aa) \triangleleft \BB \big) \triangleleft \CC \\     
   X \triangleleft \big( \Aa \otimes ( \BB  \otimes \CC ) \big) \ar[rr] 
   && (X \triangleleft \Aa) \triangleleft  (\BB  \otimes \CC ) \ , \ar[u]}
   $$
   
   $$
 \xymatrix{ X \triangleleft (\II \otimes \Aa )  \ar[rd] \ar[rr] 
 && (X \triangleleft \II) \triangleleft \Aa \ar[ld] \\
 & X \triangleleft \Aa\ ,}
 $$
 
    $$
 \xymatrix{ X \triangleleft (\Aa \otimes \II )  \ar[rd] \ar[rr] 
 && (X \triangleleft \Aa) \triangleleft \II \ar[ld] \\
 & X \triangleleft \Aa\ .}
 $$

Similarly, if $\catC$ is cotensored over $\catE$, then the construction $(\Aa,Y) \mapsto
\langle \Aa, Y \rangle$ is binatural. Thus, one gets a functor
 $$
 \langle -, - \rangle: \catE^{op} \times \catC \ra \catC .
 $$
Moreover, the composition and the unit of the enrichment
induce functorial isomorphisms
 $$
 \begin{cases}
  \langle \Aa \langle \BB, X \rangle \rangle\to\langle \Aa \otimes \BB, X \rangle \ ,\\
\langle \II, X \rangle  \simeq X\ .
\end{cases}
 $$
such that the duals of the above diagrams are commutative.

Now, consider two monoidal categories $(\catE, \otimes ,\II)$ and $(\catF, \otimes ,\II)$ and a monoidal adjunction
\[
\begin{tikzcd}
\catE \arrow[r, shift left, "L"] &  \catF\ . \arrow[l, shift left, "R"]
\end{tikzcd}
\]

Let $\CCC$ be a category enriched over $\catF$ that is tensored (resp. cotensored). Then
the category enriched over $\catF$ $R(\CCC)$ is also tensored (resp. cotensored).
The bifunctor associated to the tensoring is $-\triangleleft L (-)$ and the bifunctor associated to the cotensoring is $\langle L(-),-\rangle$.

\subsection{Model category enriched over a monoidal model category}

In this subsection $(\catE, \otimes ,\II)$ is a monoidal model category.

\begin{definition}[Homotopical enrichment]\label{defin:almosthomotop}\leavevmode
Let $\mathsf M$ be a model category. We say that  $\mathsf M$ is homotopically enriched over $\catE$ if it enriched over $\catE$ and if for any cofibration $f: X \ra X'$ in $\mathsf M$ and any fibration $g : Y \ra Y'$ in $\mathsf M$, the morphism in $\catE$:
 $$
 \{X',Y\} \ra \{X',Y\} \times_{\{X,Y'\}} \{X,Y\}
 $$ 
 is a fibration. Moreover, we require this morphism to be a weak equivalence whenever $f$ or $g$ is a weak equivalence.
\end{definition}

\begin{definition}
 A $\catE$-model category is a model category $\catM$ tensored-cotensored-enriched over $\catE$ such that the enrichment of $\catM$ over $\catE$ is homotopical.
\end{definition}

One can show that the enrichment is homotopical if for any cofibration $f: X \ra Y$ in $\mathsf M$ and any cofibration $g : A \ra B$ in $\mathsf E$, the morphism in $\catM$:
 \[
 X \triangleleft B \sqcup_{X \triangleleft A} Y \triangleleft A \to Y \triangleleft B
 \] 
 is a cofibration, and it is a weak equivalence whenever $f$ or $g$ is a weak equivalence.

\begin{proposition}\label{prop:changehomotopicalenrich}
 Consider a monoidal Quillen adjunction $\catE \dashv \catF$. If $\catM$ is an $\catF$-model category, then it has a canonical structure of a $\catE$-model category.
\end{proposition}

\begin{proof}
 We already know that $\catM$ is tensored-cotensored-enriched over $\catE$. The enrichment is homotopical because the functor $\catE \to \catF$ is a left Quillen functor.
\end{proof}
\subsection{Dwyer-Kan model structure}

This section recalls model structures on categories enriched  over a monoidal model category $\catE$ in the vein of \cite[A.3.1]{Lurie09}. There are other results related to this subject. See for instance \cite{BergerMoerdijk13}, \cite{Caviglia14}.\\

Let $(\catE,\otimes, \II)$ be a monoidal model category.

\begin{definition}
 Let $\pi_0$ be the following composite functor
 \[
 \catE \xrightarrow{\pi} Ho(\catE) \xrightarrow{S} \Set\ .
 \]
 where $Ho(\catE)$ is the homotopy category of $\catE$ and $S$ is the Yoneda functor $\hom_{Ho(\catE)}(\II, -)$. Since  the localization functor $\catE \to Ho(\catE)$ and $S: Ho(\catE) \to \Set$ are both lax monoidal, then $\pi_0$ is lax monoidal.
\end{definition}

\begin{definition}
If it exists, the Dwyer-Kan model structure on the category $\mathsf{Cat}_{\catE}$ is the model structure such that
 \begin{itemize}
 \item[$\triangleright$] the weak equivalences are the functors $F$ from $(\CCC, \mu, u)$ to $(\CCC', \mu', u')$ such that $F_{x,y}: \CCC(x,y) \to \CCC'(F(x),F(y))$ is a weak equivalence of $\catE$ for any $x,y \in Ob(\CCC)$, and such that $\pi_0 (F): \pi_0 (\CCC) \to \pi_0(\CCC')$ is an essentially surjective functor.
 \item[$\triangleright$] the (large) set of cofibrations is the smallest subset of the set of functors which is stable under pushouts, transfinite compositions and retracts and which contains the functor $\emptyset \to *_{\mathbb 1}$ and, for any cofibration $f: X \to Y$ of $\catE$, the functor $[1]_f: [1]_X \to [1]_Y$.
\end{itemize}
\end{definition}

\begin{theorem}(\cite[A.3.2.4]{Lurie09})\label{thm:modelstructure}
 Suppose that $(\catE,\otimes, \II)$ is a combinatorial monoidal model category such that
every object is cofibrant and such that weak equivalences are stable under filtered colimits. Then, the category $\catoE$ admits the Dwyer-Kan model structure which is moreover left proper and combinatorial. A set of generating cofibration is
 \[
 	\left\{ \emptyset \to *_{\mathbb 1} \right\} \sqcup \left\{  [1]_f |\ f \text{ is a generating cofibration of }\catE \right\}\ .
 \]
\end{theorem}

\begin{remark}
 Lurie assumed $\catE$ to be symmetric monoidal in his result. The proof does not relies on this assumption. Actually, this result a consequence (after some work) of the theory of combinatorial model categories ; see for instance \cite{Rosicky09}.
\end{remark}

Consider a monoidal Quillen adjunction between monoidal model categories.
\[
\begin{tikzcd}
\catE \arrow[r, shift left, "L"] &  \catF\ . \arrow[l, shift left, "R"]
\end{tikzcd}
\]

\begin{proposition}
 Suppose that the category $\catoE$ and $\mathsf{Cat}_{\catF}$ have Dwyer-Kan model structures. Suppose moreover that the functor $L: \catE \to \catF$ preserves weak equivalences (for instance if any object of $\catE$ is cofibrant). Then the adjunction
 \[
\begin{tikzcd}
\catoE \arrow[r, shift left, "L"] &  \catoF\ . \arrow[l, shift left, "R"]
\end{tikzcd}
\]
is a Quillen adjunction.
\end{proposition}

\begin{proof}
It is straightforward to prove that $R: \catoF \to \catoE$ preserves acyclic fibrations. So $L: \catoE \to \catoF$ preserves cofibrations. Let us prove that it preserves weak equivalences. Let $F: (\CCC,\gamma) \to (\DDD,\gamma)$ be weak equivalence of $\catoE$. Since the functor $L: \catE \to \catF$ preserves weak equivalences, then the map $L(F): (L \CCC)(x,y) \to (L \DDD)(F(x),F(y))$ is a weak equivalence for any objects $x,y\in Ob(\CCC)$. Besides, the functors $\pi_0: \catE \to \Set$ and $\pi_0 \circ L : \catE \to\Set$ are lax monoidal and there exists a monoidal natural transformation between them (Lemma \ref{lemmamonoidalquillenadj}). We thus obtain a natural transformation between the functor $\pi_0: \catoE \to \mathsf{Cat}$ and the functor $\pi_0 \circ L : \catoE \to \mathsf{Cat}$. So, we have the following commuting square diagram of categories
 \[
\begin{tikzcd}
	\pi_0 (\CCC) \ar[r] \ar[d,"\pi_0(F)"]
	& \pi_0 L \CCC \ar[d,"\pi_0 L(F)"] \\
	\pi_0 \DDD \ar[r]
	&\pi_0 L\DDD\ ,
\end{tikzcd}
\]
whose horizontal arrows are isomorphisms on objects. Since $\pi_0 F$ is essentially surjective, then $\pi_0 L F$ is also essentially surjective.
\end{proof}


\section{Cubical sets}

There are many different categories called cubical sets. See for instance \cite{GrandisMauri03}, and \cite{Isaacson09}. In this paper we just focus on two of these categories that we call precubical sets
and cubical sets with connections. The first one is described in \cite[\S 8.3]{Cisinski06}, in \cite{Jardine02} and in \cite{Jardine06} and the second one in \cite{Maltsiniotis09} and \cite[Example 1.6]{Cisinski14}. 

We deal with cubical sets with connections for the following reason. It has the minimal structure
that make usual nerve functors from quasi categories to enriched categories factorise through
categories enriched over cubical sets with connections.

\subsection{Segments}

We give here the notions of a segment and of an interval inspired from \cite{BergerMoerdijk06}. Note that a segment (resp. an interval) in the sense of  \cite{BergerMoerdijk06} is a monoidal segment (resp. monoidal interval) for us.

\begin{definition}
 Let $(\catE,\otimes,\II)$ be a monoidal category. A segment of $\catE$ is an object $H$ of $\catE$ together with maps
 \[
 \II \sqcup \II \xrightarrow{(\delta^0_H ,\delta^1_H)} H \xrightarrow{\sigma_H} \II
 \]
 which factorizes the morphism $ \II \sqcup \II \to \II$. Moreover, in a monoidal model category, an interval is a segment such that the map $(\delta^0_H ,\delta^1_H)$ is a cofibration and the map $\sigma_H$ is a weak equivalence.
\end{definition}

Notice that one can also define a segment as a functor $F: \Delta_{\leq 1}\to \catE$ such that $F(\Delta_0) = \II$.

\begin{definition}
Let $(\catE,\otimes,\II)$ be a monoidal category. A monoidal segment $(H,\delta^0_H,\delta^1_H,\sigma_H,\gamma_H)$ is a segment
\[
\II \sqcup \II \xrightarrow{(\delta^0_H, \delta^1_H)} H \xrightarrow{\sigma_H} \II\ ,
\]
together with a map $\gamma_H: H \otimes H \to H$ such that 
\begin{itemize}
 \item[$\triangleright$] the product $\gamma_H$ is associative, that is $\gamma_H (Id_H \otimes \gamma_H)=\gamma_H (\gamma_H\otimes Id_H)$,
 \item[$\triangleright$] the product has a unit given by $\delta^0_H: \II \to H$,
 \item[$\triangleright$] the morphism $\sigma_H$ is a morphism of monoids,
 \item[$\triangleright$] the morphism $\delta^1_H: \II \to H$ is absorbing, that is the following diagram commutes
  \[
	\begin{tikzcd}[ampersand replacement=\&]
		H \simeq \II \otimes H \arrow[r,"\delta^1_H \otimes Id" ] \arrow[d, "\sigma_H"']
		\& H \otimes H \arrow[d, "\gamma_H"]
		\& H \simeq H \otimes \II \arrow[l,"Id \otimes \delta^1_H"'] \arrow[d,"\sigma_H"] \\
		\II \arrow[r,"\delta^1_H"'] 
		\& H
		\& \II \arrow[l,"\delta^1_H"]\ .
	\end{tikzcd}
\]

\end{itemize}
A morphism of monoidal segments from $(H,\delta^0_H,\delta^1_H,\sigma_H,\gamma_H)$ to $(H',\delta^0_{H'},\delta^1_{H'},\sigma_{H'},\gamma_{H'})$ is a morphism of segment $f: H \to H'$ such that $f\gamma _H = \gamma_{H'} (f\otimes f)$. In a monoidal model category, a monoidal segment which is also an interval is called a monoidal interval.
\end{definition}

\begin{remark}
One also can deal with segments $H$ with a monoidal structure so that $\delta^0_H$ is absorbing instead
of $\delta^1_H$. This would yield another category of cubes called cubes with a left connections, while our category of cubes with connections is actually the category of cubes with right connections.
\end{remark}

\begin{remark}
One can add many different structures to segments. Any type of structured segment yields
a category of cubes which is the PRO category associated to that structured segment, and then
a category of cubical sets which are presheaves on the category of cubes. See again \cite{GrandisMauri03}.
\end{remark}

\subsection{Cubes}

This subsection deals with two categories of cubes : the category of precubes and
the categogry of cubes with connections.\\

For any integers $n \in \mbn$, we denote by $\square^n$ the $n$-times product of the poset $\{0 < 1 \}$
\[
\square^n := \{0 < 1\}^{n}\ .
\]

Notice that the full subcategory of the category of sets spanned by the objects $\square^n$ has a symmetric monoidal structure given by the cartesian product. Consider the following functions
\[
\begin{cases}
\delta^0: \square^0 \to \square^1\ ,\\
\delta^1: \square^0 \to \square^1\ ,\\
\sigma: \square^1 \to \square^0\ ,\\
\gamma : \square^2 \to \square^1\ ,
\end{cases}
\]
defined by 
\[
\begin{cases}
\delta^0(*) = 0\ ,\\
\delta^1 (*) = 1\ ,\\
\sigma (i) = *\ \forall\ i\in \square^1\ ,\\
\gamma (i,j)  = \mathrm{max} (i,j)\ \forall\ i,j\in \square^1\ .
\end{cases}
\]
Tensoring $\delta^0$, $\delta^1$, $\sigma$ and $\gamma$ with identities, one obtains the following maps
\[
\begin{cases}
 \delta^{0}_i= Id^{i} \times \delta^0 \times Id^{n-i}: \square^n \simeq \square^{i} \times \square^0 \times \square^{n-i}  \to \square^{n+1}\ ,\\
 \delta^{1}_i= Id^{i} \times \delta_1 \times Id^{n-i}: \square^n \simeq \square^{i} \times \square^0 \times \square^{n-i}  \to \square^{n+1}\ ,\\
 \sigma_{i}= Id^{i} \times \sigma \times Id^{n-i-1}: \square^{n} \simeq \square^{i} \times \square^1 \times \square^{n-i-1} \to \square^{n-1}\ ,\\
 \gamma_i = Id^{\times i} \times \gamma \times Id^{\times n-i-1} : \square^n \simeq \square^i \times \square^2 \times \square^{n-i-2} \to \square^{n-1} \ ,
\end{cases}
\]
for any $n$. The maps $\delta^{0}_i$ and $\delta^{1}_i$ are called cofaces, the maps $\sigma_i$ are called codegeneracies and the maps $\gamma^i_n$ are called connections. These functions satisfy the following relations
\[
\begin{cases}
 \delta^{\epsilon}_i \delta^{\epsilon'}_j =  \delta^{\epsilon'}_{j+1}   \delta^{\epsilon}_i\text{, if } i \leq j\ ,\\
 \sigma_{i} \sigma_{j}=  \sigma_{j}   \sigma_{i+1}\text{, if }i \geq j\ ,\\
 \sigma_i \delta^{\epsilon}_j =  \delta^{\epsilon}_{j-1} \sigma_{i} \text{, if }i<j\ ,\\
 \sigma_{i}   \delta^{\epsilon}_i=Id\ ,\\ 
 \sigma_i \delta^{\epsilon}_j =\delta^{\epsilon}_{j} \sigma_{i-1} \text{, if }i>j\ ,\\
 \gamma_i \gamma_i = \gamma_i \gamma_{i+1}\ ,\\
 \gamma_i \gamma_j = \gamma_{i+1} \gamma_j \text{, if }i>j\ ,\\  
\sigma_i \gamma_i  = \sigma_i \sigma_i\ ,\\
\sigma_i \gamma_j = \gamma_{j-1} \sigma_i \text{, if } i <j\ ,\\
\sigma_i \gamma_j = \gamma_j \sigma_{i+1} \text{, if } i >j\ ,\\
\gamma_i \delta^0_i = \gamma_i \delta^0_{i+1} =Id\ ,\\
\gamma_i \delta^1_i = \gamma_i \delta^1_{i+1} = \delta^1_i \sigma_i\ ,\\
\gamma_i \delta^\epsilon_j = \delta^\epsilon_{j-1} \gamma_i \text{, if } i+1 <j\ ,\\
\gamma_i \delta^\epsilon_j =   \delta^\epsilon_j \gamma_{i-1} \text{, if } i >j\ .
\end{cases}
\]
 
\begin{definition}\label{defin:precubes}
 The category of precubes $\square_p$ is the subcategory of posets whose objects are the posets $\square^n$ for $n \in \mbn$ and whose morphisms are generated by the cofaces $\delta^{0}_i$ and $\delta^{1}_i$ and the codegeneracies $\sigma_i$.  The category of cubes with connections $\square_c$ is the subcategory of posets whose objects are the posets $\square^n$ for $n \in \mbn$ and whose morphisms are generated by the cofaces $\delta^{0}_i$ and $\delta^{1}_i$, the codegeneracies $\sigma_i$ and the connections $\gamma_i$.
\end{definition}

\begin{remark}
The term connections was introduced in \cite{BrownHiggins81}.
\end{remark}

\begin{remark}
 Beware! The category of cubes with connections that we consider is not the same as the category considered by \cite{GrandisMauri03}, but it is the category considered in \cite{Maltsiniotis09}.
\end{remark}

The rewriting rules given above give us the following proposition.

\begin{proposition}\cite{GrandisMauri03}\label{prop:rewriting}
 Any morphism in the category $\square_p$ may be uniquely written as a sequence
 \[
 \delta^{\epsilon_1}_{i_1} \cdots  \delta^{\epsilon_n}_{i_n} \sigma_{j_1} \cdots \sigma_{j_m}\ ,
 \]
 where $i_1 > \cdots > i_i$ and $j_1 < \cdots < j_m$. Similarly, any morphism in the category $\square_c$ may be uniquely written as a sequence
 \[
 \delta^{\epsilon_1}_{i_1} \cdots  \delta^{\epsilon_n}_{i_n} \gamma_{j_1} \cdots \gamma_{j_m} \sigma_{k_1} \cdots \sigma_{k_l}\ ,
 \]
 where $i_1 > \cdots > i_i$, $j_1 \leq \cdots \leq j_m$ and $k_1 < \cdots < k_l$.
\end{proposition}

\begin{proposition}
 Both the category of precubes and the category of cubes with connections inherit a strict monoidal structure from the cartesian product of posets. In both cases, the unit is $\square^0$ and
 \[
 \square^n \otimes \square^m = \square^{n+m}\ ,
 \]
 for any integers $n,m.$
\end{proposition}

\begin{proof}
The proposition follows from long but straightforward checkings.
\end{proof}

\begin{remark}
The monoidal category $(\square_p, \otimes , \square^0)$ is \textit{not} symmetric monoidal. Indeed, the following diagram does not commute
 \[
	\begin{tikzcd}[ampersand replacement=\&]
		\square^1
		\arrow[r, "="]
		\arrow[d, "="']
		\& \square^0 \otimes \square^1 \arrow[d, "\delta^0 \otimes Id"] \\
		\square^1 \otimes \square^0 \arrow[r,"Id \otimes \delta^0"'] \& \square^1 \otimes \square^1 = \square^2\ .
	\end{tikzcd}
\]
For the same reason, the monoidal category $(\square_c,\otimes ,\square^0)$ is not symmetric monoidal.
\end{remark}

\begin{proposition}\cite[Proposition 8.4.6]{Cisinski06}, \cite[Proposition 5.5]{Maltsiniotis09}\label{prop:extensiontocubes}
Let $(\catC, \otimes , \mathbb 1)$ be a monoidal category. The category of strong monoidal functors from $\square_p$ to $\catC$ and monoidal natural transformations is equivalent to the category of segments of $\catC$. Similarly, the category of strong monoidal functors from $\square_c$ to $\catC$ and monoidal natural transformations is equivalent to the category of monoidal segments of $\catC$.
\end{proposition}

\begin{remark}
Actually Maltsiniotis shows the above results for strict monoidal categories
and strict monoidal functors. But the same arguments work.
\end{remark}

\subsection{Cubical sets}

\begin{definition}
 The category of precubical sets $\crSet$ is the category of presheaves on $\square_p$, that is functors from $\square_p^{op}$ to $\Set$. The category of cubical sets with connections $\ccSet$ is the category of presheaves on $\square_c$, that is functors from $\square_c^{op}$ to $\Set$.
\end{definition}

\begin{notation}\ 
\begin{itemize}
 \item[$\triangleright$]  We will denote by $\square_p[n]$ (resp. $\square_c[n]$) the Yoneda embedding of $\square^n$ in the category $\crSet$ (resp. $\ccSet$).
 \item[$\triangleright$] For any precubical set or any cubical set with connections $X$ and any integer $n\in \mbn$, $X(n)$ will denote the set $X(\square^n)$.
\end{itemize}
\end{notation}

\begin{example}
We will often deal with the following cubical sets.
\begin{itemize}
 \item[$\triangleright$] We denote by $\partial \square_p[n]$ the union of all the faces of $\square_p[n]$, that is
 \[
 \partial \square_p[n] (\square^m) = \{ f\in \hom_{\square_p} (\square^m , \square^n) | \text{ there is a factorisation } f=\delta_{n-1}^{\epsilon,i} g\}\ .
 \]
 \item[$\triangleright$] For any $(\epsilon,i)\in \{0,1\} \times \{1,\ldots,n\}$, let $\sqcap^{\epsilon,i}[n]$ be the $i$-cap of $\partial \square_p[n]$, that is
 \[
 \sqcap_p^{\epsilon,i}[n] (\square^m) = \{ f \in \hom_{\square_p} (\square^m , \square^n) | \text{ there is a factorisation } f=\delta_{n-1}^{\epsilon',i'} g \text{ with } (\epsilon',i') \neq (\epsilon,i)\}\ .
 \]

 \item[$\triangleright$] One can define in a similar way the cubical sets with connections $\partial \square_c[n]$ and $ \sqcap_c^{\epsilon,i}[n]$.
\end{itemize}
 \end{example}
 
 \begin{remark}
 Since the categories $\crSet$ and $\ccSet$ are presentable and since the bifunctor $- \otimes -$ commutes with colimits on both sides, then one can show that the monoidal categories $\crSet$ and $\ccSet$ are biclosed, that is the functor $-\otimes X$ and the functor $X \otimes - $ have both right adjoints.
\end{remark}

\begin{proposition}\cite[8.4.23]{Cisinski06}\label{prop:extension}
 Let $(\catC, \otimes , \mathbb 1)$ be a monoidal cocomplete bilinear category. The category of cocontinuous strong monoidal functors from $\crSet$ to $\catC$ with monoidal natural transformations is equivalent to the category of segments of $\catC$. Similarly, the category of cocontinuous strong monoidal functors from $\ccSet$ to $\catC$ with monoidal natural transformations is equivalent to the category of monoidal segments of $\catC$.
\end{proposition}

\begin{proof}
 It is a straightforward consequence of Proposition \ref{prop:extensiontocubes}.
\end{proof}

Moreover, any functor $L_p : \crSet \to \catC$ which is cocontinuous has a right adjoint. Then, any adjunction
\[
\begin{tikzcd}
\crSet \arrow[r, shift left, "L_p"] &  \catC\ , \arrow[l, shift left, "R_p"]
\end{tikzcd}
\] 
is essentially determined by the image under the functor $L_p$ of $\square_p[1]$. The same result holds if we replace the category $\crSet$ by the category $\ccSet$.

\begin{notation}
 Let $(\catC, \otimes , \II)$ be a bilinear monoidal category. Consider a segment $H$ in $\catC$. The adjunction relating $\catC$ to precubical sets induced by this segment will be denoted $L^H_p \dashv R^H_p$. If $H$ is a monoidal segment, the induced adjunction relating $\catC$ to cubical sets with connections  will be denoted $L^H_c\dashv R^H_c$.
\end{notation}

\begin{remark}
 Notice that for any monoidal segment $H$,
 \[
 L^H_c \circ   L^{\square_c[1]}_p \simeq L^H_p\ .
 \]
\end{remark}

\begin{lemma}\cite[Lemme 8.4.36]{Cisinski06}
 For any integer $n \in \mbn$ and for any integers $i,j\in \mbn$ such that $i+j=n$, we have
 \[
 \partial \square_p[n] \simeq \partial \square_p[i] \otimes \square_p[j] \sqcup_{\partial \square_p[i] \otimes \partial \square_p[j]} \square_p[i] \otimes \partial \square_p[j]\ .
 \]
 Moreover,
\begin{align*}
 \sqcap_p^{\epsilon,n}[n] &\simeq \partial \square_p[n-1] \otimes \square_p[1] \sqcup_{ \partial \square_p[n-1] \otimes \square_p[0]} \square_p[n-1] \otimes \square_p[0]\\
  \sqcap_p^{\epsilon,1}[n] &\simeq \square_p[1] \otimes \partial \square_p[n-1] \sqcup_{ \square_p[0] \otimes \partial \square_p[n-1]} \square_p[0] \otimes \square_p[n-1]\\
  \sqcap_p^{\epsilon,i}[n] &\simeq \sqcap_p^{\epsilon,i}[j] \otimes \square_p[j] \sqcup_{\sqcap_p^{\epsilon,i}[i] \otimes \partial \square_p[j]}  \square_p[i]\otimes \partial \square_p[j]\ .
\end{align*}
 The same results hold in the category of cubical sets with connections.
\end{lemma}

\poubelle{
\subsection{Skeleton and coskeleton}

\begin{definition}
 For any integer $n\in \mbn$, let $\square^{\leq n}$ be the full subcategory of the category $\square$ of cubes spanned by the objects $(\square^k)_{k \leq n}$. We denote by $\square^{\leq n}-\mathsf{set}$ the category of presheaves on $\square^{\leq n}$.
\end{definition}

The inclusion functor $\square^{\leq n} \hookrightarrow \square$ induces a restriction functor $r_n:\cSet \to \square^{\leq n}-\mathsf{set}$. This restriction functor preserves both colimits and limits. By the adjoint functor theorem it has both a left and a right adjoint.

\begin{definition}
 The $n^{th}$ coskeleton $\mathrm{cosk}_n$ (resp. the $n^{th}$ skeleton $\mathrm{sk}_n$) is the left (resp. right) adjoint functor of the functor $r_n$.
\end{definition}

\begin{lemma}\label{lemma:cosk}
 For any integers $n \leq m$, then $\mathrm{cosk}_m(\square[n]) \simeq \square[n]$.
\end{lemma}

\begin{proof}
We have
\begin{align*}
 \hom_{\cSet}(\square[n], X) &\simeq X_n \simeq \left(r_m(X)\right)_n\\ & \simeq \hom_{\square^{\leq m}-\mathsf{set}}(\square[n], r_m(X))\\
&\simeq   \hom_{\cSet}(\mathrm{cosk}_m(\square[n]), X)\ ,
\end{align*}
for any cubical set $X$. So, by the Yoneda lemma, $\mathrm{cosk}_m(\square[n]) \simeq \square[n]$.
\end{proof}

It is well known that such functors $r_n$, $\mathrm{cosk}_n$ and $\mathrm{sk}_n$ exist in the context of simplicial sets.}


\section{Homotopy theory of cubical sets}

This section deals with the homotopy theories of precubical sets and of cubical sets with connections. We first recall their model structures from the work of Cisinski, Jardine and Maltsiniotis and give additional results on the model structure on cubical sets with connections. We then study the Reedy model structure on cocubical objects and its link to functors from cubical sets (precubical or with connections) to any other monoidal model category.

\subsection{Homotopy theory of precubical sets}
We know that the unit of the cartesian monoidal structure on simplicial sets is the final object $\Delta[0]$. Moreover, the map $\Delta[0] \sqcup  \Delta[0] \to \Delta[0]$ may be factorised as follows
\[
\Delta[0] \sqcup \Delta[0] \xrightarrow{\delta^0 \sqcup \delta^1} \Delta[1] \xrightarrow{\sigma} \Delta[0]\ .
\]
By Proposition \ref{prop:extension}, this induces a strong monoidal cocontinuous functor $L^{\Delta[1]}_p: \crSet \to \sSet$ and hence a monoidal adjunction between precubical sets and simplicial sets
 \[
\begin{tikzcd}
\crSet \arrow[r, shift left, "L^{\Delta[1]}_p"] &  \sSet \arrow[l, shift left, "R^{\Delta[1]}_p"]\ .
\end{tikzcd}
\]
One can transfer the usual model structure on simplicial sets to precubical sets.

\begin{theorem}\cite{Cisinski06}\cite{Jardine02}
\label{thm: after cisinski jardine}
 There exists a combinatorial proper model structure on precubical sets such that
\begin{itemize}
\item[$\triangleright$] the cofibrations are the monomorphisms,
\item[$\triangleright$] the weak equivalences are the morphisms $f$ such that $L_p^{\Delta[1]} (f)$ is a weak equivalence,
\item[$\triangleright$] the generating cofibrations are the injections $\partial(\square)[n] \hookrightarrow \square[n]$,
\item[$\triangleright$] the generating acyclic cofibrations are the injections $\sqcap^{\epsilon,i}[n] \hookrightarrow \square[n]$,
\item[$\triangleright$] the adjunction $L_p^{\Delta[1]} \dashv R_p^{\Delta[1]}$ is a Quillen equivalence,
\item[$\triangleright$] $(\crSet,\otimes ,\square^0)$ is a monoidal model category,
\item[$\triangleright$] the functor $R_p^{\Delta[1]}$ preserves and reflects weak equivalences.
\end{itemize}
\end{theorem}

\begin{remark}
 Jardine and Cisinski described independently two model structures on the category of cubical sets ; see \cite{Jardine02} and \cite{Cisinski06}. They are actually the same (see \cite{Jardine06}). Besides, the fact that this model structure is right proper is a direct consequence of Cisinski's theory. However, this property seems hard to prove in Jardine's framework which is more topological.
\end{remark}

\subsection{Homotopy theory of cubical sets with connections}

\begin{theorem}(\cite[Proposition 3.3]{Maltsiniotis09} and \cite[Theorem 1.7]{Cisinski14})\label{thm:aftercisinski}
The category $\square_c$ is a test category. Hence, the category $\ccSet$ admits a model structure
whose cofibrations are monomorphisms and whose weak equivalences are maps $f:X \to Y$
 such that the map
\[
	N(\square_c/X) \to N(\square_c/Y) 
\] 
is a weak equivalence of simplicial sets for the Kan-Quillen model structure
and so that the functor from $\ccSet$ to $\sSet$ that sends $X$ to $N(\square_c/X)$
is an equivalence at the level of homotopy categories.
Moreover, the model structure is monoidal, combinatorial and proper and
\begin{itemize}
 \item[$\triangleright$] a set of generating cofibrations is given by the maps
 \[
 \{ \partial \square_c[n]  \to \square_c[n] |n \in \mbn\}\ .
 \] 
 \item[$\triangleright$] a set of generating acyclic cofibrations is given by the maps
 \[
 \{ \sqcap^{i,\epsilon}_c[n]  \to \square_c[n] |n \in \mbn\}\ .
 \] 
\end{itemize}
\end{theorem}

\begin{remark}
Actually, Maltsiniotis showed that $\square_c$ is a strict test
 category.
\end{remark}

\subsection{The Reedy model structure on cocubical objects}

In this subsection and until the end of this section (except the last subsection "Cubical sets with connections and cubical sets"), what we describe holds in the category of precubical sets and in the category of cubical sets with connections. Therefore, we will not use the indices $c$ or $r$ but use the notation $\square$ and talk about cubical sets. 

\begin{proposition}\label{prop: reedy struct}
 The category $\square$ has a Reedy structure such that
\begin{itemize}
 \item[$\triangleright$] the degree of $\square^n$ is $n$,
 \item[$\triangleright$] the degree raising morphisms are the composites of cofaces,
 \item[$\triangleright$] the degree lowering morphisms are the composites of codegeneracies (or the composites of codegeneracies and connections).
\end{itemize}
\end{proposition}

Let $\mathsf{E}$ be a model category. We know that the category $\mathrm{Fun}(\square, \mathsf E)$ of cocubical objects in $\mathsf{E}$ is equivalent to the category $\mathrm{Fun}_{cc}(\cSet, \mathsf E)$ of functors from cubical sets to $\mathsf E$ which preserve colimits. Therefore, we will often assimilate a cocubical object to such a functor. We know that  a cocubical object $F$ also induces a functor $F^!$ from $\mathsf E$ to cubical sets defined by 
\[
F^!(X):= \hom_{\catE} (F(-), X)\ ,
\]
and which is right adjoint to $F_!$. Besides a map $F\to G$ induces a natural transformation $G^! \to F^!$.\\

We can endow the category $\mathrm{Fun}(\square, \mathsf E)$ of cocubical objects in $\mathsf{E}$ with the Reedy model structure (see for instance \cite[Theorem 5.2.5]{Hovey99}), that is,
\begin{itemize}
 \item[$\triangleright$] the weak equivalences are the morphisms $F\to G$ such that $F(\square[n]) \to G(\square[n])$ is a weak equivalence in $\mathsf E$ for any $n \in \mbn$,
  \item[$\triangleright$] the cofibrations (resp. acyclic cofibrations) are the morphisms $F\to G$ such that  
 \[
F(\square[n]) \sqcup_{F(\partial \square [n])}  G(\partial \square [n]) \to G(\square[n])
 \]
is a cofibration (resp. an acyclic cofibration) in $\mathsf E$ for any $n \in \mbn$.
\end{itemize}

\begin{proposition}\label{prop:reedy2}
 Let $F \to G$ be a Reedy cofibration of cocubical objects of $\catE$ and let $p: X \to Y$ be a cofibration of $\catE$. Then if one these two maps is also a weak equivalence, then the morphism
 \[
G^!(X) \to G^!(Y) \times_{F^!(Y)} F^!(X)
 \]
 is an acyclic fibration.
\end{proposition}

Equivalently, for any Reedy cofibration $F\to G$ and for any cofibration of cubical sets $A \to B$, the morphism in $\catE$
\[
 F(B) \sqcup_{F(A)} G(A) \to G(B) 
 \]
is a cofibration and it is an acyclic cofibration if  $F \to G$ is acyclic. In particular, a Reedy cofibrant functor $F$ preserves cofibrations.

\begin{proof}[Proof of Proposition \ref{prop:reedy2}]
 Consider a square diagram as follows
 \[
\begin{tikzcd}
 \partial (\square[n])  \arrow[r] \arrow[d] & G^!(X) \arrow[d]\\
 \square[n] \arrow[r] & G^!(Y) \times_{F^!(Y)} F^!(X)\ .
\end{tikzcd}
 \]
 It induces another square diagram in $\catE$
\[
\begin{tikzcd}
 F(n) \sqcup_{\partial F(n)} \partial G(n)  \arrow[r] \arrow[d] & X \arrow[d]\\
 G(n) \arrow[r] & Y\ .
\end{tikzcd}
 \]
This square diagram has a lifting because one of the vertical maps is a weak equivalence. So the first square diagram has also a lifting.
\end{proof}

\begin{corollary}\label{cor:reedy}
 Let $F$ and $G$ be two Reedy cofibrant cocubical objects of $\mathsf E$ and consider a weak equivalence $F\to G$. Then, for any cubical set $A$, the morphism $F_!(A) \to G_! (A)$ is a weak equivalence and for any fibrant object $X$ of $\catE$, the morphism $G^!(X) \to F^!(X) $ is a weak equivalence.
\end{corollary}

\begin{proof}[Proof of Corollary \ref{cor:reedy}]
It is a straightforward consequence of K. Brown's lemma (functors that preserves acyclic cofibrations
preserve weak equivalences between cofibrant objects) and Proposition \ref{prop:reedy2}.
\end{proof}

\subsection{Intervals and Quillen adjunctions}

Let $(\catE,\otimes , \mathbb 1)$ be a monoidal model category. Let us choose a segment $\II \sqcup \II \to H \to \II$ (or a monoidal segment in the case of cubical sets with connections). We know that such a segment induces a monoidal adjunction $L^H \dashv R^H$ relating cubical sets to $\catE$.

\begin{proposition}\label{prop:interval}
 The adjunction $L^H \dashv R^H$ is a Quillen adjunction if and only if $H$ is an interval, that is the morphism $\II \sqcup \II \to H$ is a cofibration and the morphism $H \to \II$ is a weak equivalence.
\end{proposition}

\begin{remark}
This is a cubical analogue of \cite[Prop. A.13]{BergerMoerdijk06}.
\end{remark}

\begin{lemma}\label{lemma:lhreedy}
 The functor $L^H$ preserves cofibrations if and only if the map $\II \sqcup \II \to H$ is a cofibration.
\end{lemma}

\begin{proof}
If $L^H$ preserves cofibrations, then the map
\[
	\II \sqcup \II \simeq L^H (\partial \square[1]) \to  L^H (\square[1]) \simeq  H
\]
is a cofibration. Conversely, suppose that the map $\II \sqcup \II \to H$ is a cofibration. Let us prove by induction that, for any integer $n \in \mbn$, the map
\[
L^H (\partial \square [n]) \to L^H (\square [n])
\]
is a cofibration. Since $\catE$ is a monoidal model category, the map $\emptyset \to \II$ is a cofibration. So the result holds for $n=0$. Suppose that it holds for some integer $n$. Then the morphism
 \[
 L^H (\partial \square[n+1] ) \simeq (\II \sqcup \II) \otimes L^H (\square[n]) \sqcup_{(\II \sqcup \II) \otimes L^H (\partial \square [n])} H \otimes L^H (\partial \square [n]) \to H \otimes  L^H (\square [n]) \simeq L^H(\square[n+1])
 \]
 is a cofibration. So the result holds at the stage $n+1$.
\end{proof}

\begin{proof}[Proof of Proposition \ref{prop:interval}]
Suppose that $H$ is an interval and let us prove that $L^H$ is left Quillen. We already know from Lemma \ref{lemma:lhreedy} that it preserves cofibrations. So it suffices to show that it preserves acyclic cofibrations. Let $n$ be an integer, let $0\leq i \leq n$ and let $\epsilon\in \{0,1\}$. We will denote the opposite sign of $\epsilon$ by $\ov \epsilon$. Then $L^H(\sqcap^{i,\epsilon}[i])$ is the colimit of the following diagram
  \[
\begin{tikzcd}[ampersand replacement=\&]
	L^H(\partial(\square[i-1])) \otimes \II \arrow[r]\arrow[d]  \&L^H(\partial(\square[i-1])) \otimes H\\
	L^H(\square[i-1]) \otimes \II
\end{tikzcd}
 \]
 Since $L^H(\partial(\square[i-1]))$ is cofibrant by Lemma \ref{lemma:lhreedy}, and since the map $\II\to H$ is an acyclic cofibration, then the morphism $L^H(\partial(\square[i-1])) \otimes \II \to L^H(\partial(\square[i-1])) \otimes H$ is an acyclic cofibration.
 So, the map $L^H(\square[i-1]) \to L^H(\sqcap^{i,\epsilon}[i])$ is also an acyclic cofibration. Besides, since $H^{\otimes i-1}$ is cofibrant and since the map $\II\to H$ is an acyclic cofibration, then the map
 \[
 \delta_i^{i, \ov \epsilon}:L^H(\square[i-1]) \simeq H^{\otimes i-1} \otimes \II  \rightarrow H^{\otimes i} \simeq  L^H(\square[i])
 \]
 is a weak equivalence. So the map $L^H(\sqcap^{i,\epsilon}[i]) \to L^H(\square[i])$ is a weak equivalence. It is even an acyclic cofibration since the map $\sqcap^{i,\epsilon}[i] \to\square[i]$ is a cofibration and since $L^H$ preserves cofibrations. Besides, $L^H(\sqcap^{i,\epsilon}[n])$ is the colimit of the following diagram
  \[
\begin{tikzcd}[ampersand replacement=\&]
	L^H(\sqcap^{i,\epsilon}[i]) \otimes L^H(\partial(\square[n-i])) \arrow[r,hook,"\sim"]\arrow[d,hook]  \&L^H(\square[i]) \otimes L^H(\partial(\square[n-i]))\\
	L^H(\sqcap^{i,\epsilon}[i]) \otimes L^H(\square[n-i])\ .
\end{tikzcd}
 \]
 Using the same arguments as in the paragraph just above, we can prove that the map $L^H(\sqcap^{i,\epsilon}[n]) \to L^H (\square[n])$ is an acyclic cofibration. So $L^H$ preserves acyclic cofibrations. Therefore, it is a left Quillen functor. The converse implication is straightforward.
\end{proof}

Consider a morphism if intervals $H \to H'$. By Proposition \ref{prop:extension}, it induces a unique natural transformation $L^H \to L^{H'}$. Thus, we obtain a natural transformation $R^{H'} \to R^H$.

\begin{proposition}\label{prop:changeint}
Consider a morphism of interval $H \to H'$. This is in particular a weak equivalence. Then, for any fibrant object $X$ on $\mathsf{E}$, the morphism
$
R^{H'}(X) \to R^H(X)
$
is a weak equivalence. Moreover, for any cubical set $X$, the morphism $L^{H}(X) \to L^{H'}(X)$ is a weak equivalence.
\end{proposition}

\begin{proof}[Proof of Proposition \ref{prop:changeint}]
By Proposition \ref{prop:interval}, we know that the functor $L^H$ and $L^{H'}$ are both Reedy cofibrant. Moreover, the map $L^H \to L^{H'}$ is an equivalence. We conclude by Corollary \ref{cor:reedy}.
\end{proof}

\begin{corollary}\label{corquilleneq}
The two following propositions are equivalent. 
\begin{enumerate}
 \item There exists an interval $H$ of $\catE$ such that the adjunction $L^H\dashv R^H$ is a Quillen equivalence.
 \item For any interval $H$ of $\catE$, the adjunction $L^H\dashv R^H$ is a Quillen equivalence.
\end{enumerate}
\end{corollary}

\begin{proof}
 The second statement implies the first one since there exists intervals. Besides, suppose that the first statement is true, and let $H'$ be an other interval. There exists a sequence of weak equivalences of intervals as follows
 \[
 \begin{tikzcd}
	H \arrow[r,"\sim"] & H''  & H' \arrow[l,"\sim"']\ .
\end{tikzcd}
 \]
This induces a sequence of natural transformations of functors from cubical sets to $\text{Ho}(\catE)$
\[
 \begin{tikzcd}
	L^H \arrow[r] & L^{H''}  & L^{H'} \arrow[l]\ .
\end{tikzcd}	
\]
By Proposition \ref{prop:changeint}, these natural transformations are isomorphisms. So $L^{H'}$ is an equivalence of categories as well as $L^H$.
\end{proof}

\subsection{Two ways from cubical sets and simplicial sets}

From a cubical set $X$, one can produce a simplicial set in many ways. We are interested in the constructions $X \mapsto N(\square/ X)$ and  $X \mapsto L^{\Delta[1]} X$. The goal of this subsection is to show that they are naturally equivalent. result.

\begin{definition}
Let $Sd : \square_c \to \sSet$ be the functor that sends $\square[n]$ to the nerve of its category of subobjects in $\square_c$. This induces and adjunction $Sd_! \dashv Sd^\ast$ relating cubical sets to simplicial sets.
\end{definition}

\begin{definition}
 Let $\mathrm{cospan}$ be the following interval in simplicial sets
$$
\mathrm{cospan} = N(\{0 \to 2 \leftarrow 1\}).
$$
whose structure maps $\delta^0, \delta^1: \ast \to \mathrm{cospan}$ have image respectively $0$ and $1$.
\end{definition}

\begin{lemma}
We have a canonical isomorphism
$$
Sd_! \simeq L^{\mathrm{cospan}} .
$$
\end{lemma}

\begin{proof}
One can notice that monomorphisms are compositions of cofaces (both for $\square_p$ and $\square_c$). This implies that the functor $S_d$ is monoidal. Thus, one gets a canonical isomorphism
$$
Sd_! \simeq L^{Sd(\square[1])} .
$$
To conclude, one has also a canonical isomorphism of interval
$$
Sd(\square[1]) \simeq \mathrm{cospan}.
$$
Hence, one gets natural isomorphisms
$$
Sd_! \simeq L^{Sd(\square[1])} \simeq L^{\mathrm{cospan}}.
$$
\end{proof}

\begin{lemma}\cite[Corollaire 3.2]{Cisinski06}
The functor $X \mapsto N(\square / X)$ preserves colimits.
\end{lemma}

\begin{definition}
Let us denote $G: \square \to \sSet$ the functor
$$
\square[n] \mapsto G(\square[n]) = N(\square / \square[n]) .
$$
Then, one has a natural isomorphism
$G_!(X) \simeq N(\square / X)$.
\end{definition}

\begin{definition}
Let us denote $Im$ the natural transformation $G \to Sd$ so that $Im(\square[n])$ is the morphism of simplicial sets induced by the functor 
$$
\square / \square[n] \to \text{Sub-objects}(\square[n])
$$
that sends a morphism $\phi : \square[m] \to \square[n]$ to its image.
\end{definition}

\begin{definition}
Let us denote $Tr$ the monoidal natural transformation $Sd_! = L^{\mathrm{cospan}} \to L^{\Delta}$ induced by the morphism of interval
$$
\mathrm{cospan} \to \Delta[1]
$$
induced by the functor that sends $1, 2$ to $1$.
\end{definition}

\begin{lemma}\label{lemmareedycof}
The cocubical object $G \in Fun{\square}{\sSet}$ is Reedy cofibrant.
\end{lemma}

\begin{proof}
It suffices to notice that the functor $X \mapsto N(\square / X)$ preserves monomorphisms.
\end{proof}

\begin{proposition}\label{lemmachsdid}\label{propositionnateq}
The naturalation $Im_! \circ Tr_! : G_! \to L^{\Delta[1]}$ is an objectwise weak equivalence.
\end{proposition}

\begin{proof}
We known that both cubical simplicial sets $G$ and $L^{\Delta[1]}$ are Reedy cofibrant (by Lemma \ref{lemmareedycof} and Proposition \ref{prop:interval}). Moreover, for any cube $\square[n]$, the map $Tr \circ Im(\square[n])$ is a weak equivalence since both $G(\square[n])$ and $L^{\Delta[1]}(\square[n])$ are contractible.
We conclude by Corollary \ref{cor:reedy}.
\end{proof}

\subsection{Cubical sets are equivalent to simplicial sets}

Consider the following sequence of adjunctions
\[
\begin{tikzcd}
\crSet \arrow[r, shift left, "L^{\square_c[1]}_p"] &  \square_c -\mathsf{Set} \arrow[l, shift left, "R^{\square_c[1]}_p"]\arrow[r, shift left, "L^{\Delta[1]}_c"] &  \sSet \arrow[l, shift left, "R^{\Delta[1]}_c"]\ .
\end{tikzcd}
\]

\begin{proposition}\label{prop: aftercisinski reflect}
The functor $L^{\Delta[1]}_c : \ccSet \to \sSet$ preserves and reflects weak equivalences, that is
weak equivalences are maps $f$ such that $L^{\Delta[1]}_c f$ is a weak equivalence.
\end{proposition}

\begin{proof}
This is a consequence of Proposition \ref{propositionnateq}.
\end{proof}

\begin{corollary}\label{cor: aftercisinski}
The functor $L^{\square_c[1]}_p : \crSet \to \ccSet$ preserves and reflects weak equivalences.
\end{corollary}

\begin{proof}
It follows from the facts that $L^{\Delta[1]}_p = L^{\Delta[1]}_c \circ L^{\square_c[1]}_p$ 
and that $L^{\Delta[1]}_p$ and $L^{\Delta[1]}_c$ preserve and reflect weak equivalences.
\end{proof}

\begin{proposition}\label{prop: aftercisinski quillen}
The adjunction $L^{\Delta[1]}_c \dashv R^{\Delta[1]}_c$ is a Quillen equivalence as well as the adjunction $L^{\square_c[1]}_p \dashv R^{\square_c[1]}_p$.
\end{proposition}

\begin{proof}
Since both $L^{\square[1]}_p$ and $L^{\Delta[1]}_c$ send generating cofibrations to cofibrations they preserve cofibrations. Moreover since they preserve weak equivalences (Proposition \ref{prop: aftercisinski reflect}
and Corollary \ref{cor: aftercisinski}), they are left Quillen functors.
Since cubical sets with connections form a test category, the functor
 \[
 X \in \ccSet \mapsto N(\square_c/X) \in \sSet
 \]
 induces an equivalence of categories between the homotopy category of cubical sets with connections and the homotopy category of simplicial sets. Since the morphism $(Tr \circ Im)(X):N(\square_c/X)\to L^{\Delta[1]}_p(X)$ (Proposition \ref{propositionnateq}) is an equivalence for any object $X$, then $L^{\Delta[1]}_c$ is an equivalence of categories at the level of homotopy categories. Thus, $L^{\Delta[1]}_c \dashv R^{\Delta[1]}_c$ is a Quillen equivalence. Since
 $L^{\Delta[1]}_p$ is also an equivalence at the level of homotopy categories, by the 2-out-of-3 rule, this is also the case for $L^{\square_c[1]}_p$. Thus $L^{\square_c[1]}_p \dashv R^{\square_c[1]}_p$ is also a Quillen equivalence.
\end{proof}

\subsection{Cubical model categories}

\begin{proposition}\label{propcubicalmodelcat}
  Let $(\mathsf E, \otimes, \mathbb 1)$ be a monoidal model category and let $\catM$ be an $\catE$-model category. Then, any choice of an interval (resp. monoidal interval) $H$ in $\catE$ induces a structure of a $\crSet$-model category (resp. $\ccSet$-model category) on $\catM$.
\end{proposition}

\begin{proof}
 This is a consequence of Proposition \ref{prop:changehomotopicalenrich}.
\end{proof}


\section{From quasi-categories to enriched categories}

In this section, we study the link between quasi-categories and categories enriched in cubical sets. Then, when $\catE$ is a monoidal model category equipped with an interval, this allows us to give precise conditions making an adjunction relating the Joyal category of simplicial sets to the category $\catoE$ to be a Quillen adjunction.

\begin{notation}
The category of categories enriched over $\crSet$ and the category of categories enriched over $\ccSet$ are denoted respectively $\catocrSet$ and $\catoccSet$.
\end{notation}

\subsection{From cubical categories to quasi-categories}

\begin{definition}
Let $\square_c^+$ be the category obtained from the category of
cubes with connections $\square$ by adding an empty cube $\emptyset$ so
that
\[
\begin{cases}
	\hom_{\square_c^+} (\emptyset,\emptyset) =\ast,
	\\
	\hom_{\square_c^+} (\emptyset,\square^n) =\ast,
	\\
	\hom_{\square_c^+} (\square^n, \emptyset) =\emptyset.
\end{cases}
\]
In particular, $\square_c$ is the full subcategory of 
$\square_c^+$ spanned by all objects except $\emptyset$.
Moreover, one can extend the monoidal structure of $\square_c$
by $\emptyset \otimes X = X \otimes \emptyset = \emptyset$.
\end{definition}

\begin{definition}
For any integer $n \in \mbn$, let $(W_n, \mu, u)$ be the following category enriched
in
the extended category of cubes with connections $\square_c^+$
\begin{itemize}
 \item[$\triangleright$] its set of objects is $\{0,\ldots,n\}$,
 \item[$\triangleright$] for any $i<j \in \{0,\ldots,n\}$,
 \[
 W_n (i,j) = \square^{j-i-1} ;
 \]
 moreover, $W_n (i,i) = \ast$ and $W_n (j,i)= \emptyset$;
 \item[$\triangleright$] the composition is defined for any $i<j<k$ by
 \[
 \begin{tikzcd}
  W_n(i,j) \times W_n(j,k) \ar[d, "\simeq"]
\\  
  W_n(i,j) \times \ast \times W_n(j,k) 
   \ar[d,"{\mathrm{Id} \times \{1\} \times \mathrm{Id}}"]
   \\
   \square^{j-i-1} \otimes \square^1 \otimes \square^{k-j-1}
   \ar[d, equal]
   \\
   W_n^{pos}(i,k)  \ .
 \end{tikzcd} 
 \]
\end{itemize}
Since, the fully faithful inclusion $\square_c^+ \to \ccSet$ is strict monoidal, one can
also consider $W_n$ as a category enriched in cubical sets with connections.
\end{definition}

\begin{proposition}
The assignment $n \mapsto W_n$ defines a functor from the category $\Delta$ to the category $\mathsf{Cat}_{\square_c^+}$ of categories enriched in $\square_c^+$ and so to the category $\catoccSet$
which contains $\mathsf{Cat}_{\square_c^+}$ as a full subcategory.
\end{proposition}

\begin{proof}
Any coface morphism $\delta^\Delta_i: [n] \to [n+1]$ in the category $\Delta$ induces a functor $W_n \to W_{n+1}$ which is the function $\delta^\Delta_i$ on objects and such that for any $j <i \leq k$ the morphism $W_n (j,k) \to W_{n+1} (j,k+1)$ is given by
\[
\square^{k-j-1} \xrightarrow{\delta^{0}_{i-j-1}} \square^{k-j}\ .
\]
Similarly, any codegeneracy morphism $\sigma^\Delta_i: [n] \to [n-1]$ in the category $\Delta$ induces a cubical functor $W_n \to W_{n-1}$ which is the function $\sigma^\Delta_i$ on objects and such that for any $j < i < k$ the morphism $W_n (j,k) \to W_{n-1} (j,k-1)$ is given by
\[
\square^{k-j-1} \xrightarrow{\gamma_{i-j-1}} \square^{k-j-2}\ .
\]
Moreover, the morphism $W_n (i,k) \to W_{n-1} (i,k-1)$ is $\sigma_0: \square^{k-i-1} \to \square^{k-j-2}$ and the morphism $W_n (j,i+1) \to W_{n-1} (j,i)$ is $\sigma_{j-i-1}: \square^{j-i} \to \square^{j-i-1}$.
\end{proof}

\begin{definition}
 Let $W_c \dashv N^{c}$ be the adjunction relating simplicial sets to $\ccSet$-enriched-categories such that $W_{c}$ is the left Kan extension of the cosimplicial object $n \mapsto W_n$ and 
 \[
 	N^{c} (\mathcal C)(n) = \hom_{\catoccSet} (W_n, \mathcal C)\ .
 \]
\end{definition}

\begin{proposition}
The adjunction $L_c^{\Delta[1]} \circ W_{c} \dashv N^{c} \circ R_c^{\Delta[1]}$ is canonically isomorphic to the adjunction $\mathfrak{C}\dashv N$ of the book Higher Topos Theory \cite[\S 1.1.5]{Lurie09}, in the sense that $L_c^{\Delta[1]} \circ W_{c}$ is canonically isomorphic to $\mathfrak{C}$.
\end{proposition}

\begin{proof}
It suffices to exhibit a canonical isomorphism of cosimplicial simplicial categories between
$n \mapsto L_c^{\Delta[1]} (W_n) $ and $n \mapsto \mathfrak{C}([n])$.
On the one hand, let us notice that the following diagram of monoidal categories and strong monoidal
functors commute
\[
\begin{tikzcd}
	\square_c^+
	\ar[r] \ar[d,"i"] 
	& \ccSet \ar[d]
	\\
	\mathsf{Cat} \ar[r,"N"']
	&\sSet
\end{tikzcd}
\]
(up to a canonical monoidal isomorphic natural transformation)
where the left vertical functor $i$ is the inclusion
of $\square_c^+$ into posets which are particular categories.
This gives us the following commutating square
\[
\begin{tikzcd}
	\square_c^+ \text{ enriched cats}
	\ar[r] \ar[d,"i"] 
	& \catoccSet \ar[d]
	\\
	\mathsf{Cat} \text{ enriched cats} \ar[r,"N"']
	&\catosSet
\end{tikzcd}
\]
The cosimplicial simplicial category $\mathfrak{C}[-]$ is defined in \cite[\S 1.1.5]{Lurie09} as
$\mathfrak{C}[n] = N(F[n])$
for some cosimplicial category enriched in categories $F[n]$ that is canonically
isomorphic to $i(W_n)$. Thus, we get canonical isomorphisms of cosimplicial
simplicial categories
\[
	L_c^{\Delta[1]} (W_n)\simeq Ni(W_n) \simeq NF[n] = \mathfrak{C}[n] . 
\]
\end{proof}

\begin{remark}\label{newremark}
 Note that the construction above already appeared in \cite{RiveraZenalian16}. They showed that the adjunction $W_{c} \dashv N^{c}$ factors the adjunction $\mathfrak{C}\dashv N$ of the book Higher Topos Theory \cite[\S 1.1.5]{Lurie09} as well as the dg nerve of \cite[\S 1.3.1]{Lurie12}.
The main difference between this work and ours is that the factorisation provided here is
formal and does not rely on any combinatorial computation.
\end{remark}

\begin{definition}\cite{Joyal00}\cite{Lurie09}
 The Joyal model category $\sSet_J$ is the category of simplicial sets $\sSet$ equipped with the model structure whose cofibrations are monomorphisms and weak equivalences  are maps $f$
 such that $\mathfrak{C}(f)$ is a weak equivalence for the Dwyer-Kan model structure on $\catosSet$. The fibrant objects are the quasi-categories. Moreover, this is a cartesian closed monoidal model category.
\end{definition}

\begin{lemma}\label{lemmareflects}
A morphism $F: \CCC \to \DDD$ in $\catoccSet$ is a Dwyer-Kan equivalence if
and only if  $L_c^{\Delta[1]}(F)$ is a Dwyer-Kan equivalence.
\end{lemma}

\begin{proof}
On the one hand, any map $F_{x,y}: \CCC(x,y) \to \DDD(x,y)$ is an equivalence if and only if
$L_c^{\Delta[1]}(F_{x,y})$ is an equivalence by Proposition \ref{prop: aftercisinski reflect}. On the other hand, the functor $\pi_0: \ccSet \to \Set$ factorises as
\[
	\ccSet \xrightarrow{L_c^{\Delta[1]}} \sSet \xrightarrow{\pi_0} \Set.
\]
Thus, $\pi_0(F)$ is an equivalence of categories if and only if $\pi_0 L_c^{\Delta[1]}(F)$ is an equivalence since the two functors are the same.
\end{proof}

\begin{lemma}\label{lemmawpreservecof}
 The functor $W_c$ preserves cofibrations.
\end{lemma}

Equivalently, the cosimplicial cubical category $n \to W_n$ is Reedy cofibrant.

\begin{proof}
 It suffices to show that for any integer $n$, the map $W_c(\partial \Delta[n]) \to W_n$
 is a cofibration. For $n=0$, this is just the fact that
 $\emptyset \to \ast$ is a cofibration. For $n>0$, this follows from the fact that the following square of cubical categories is a pushout
 \[	
	\begin{tikzcd}
 		{[1]}_{\partial \square_c[n-1]} \arrow[r] \arrow[d,"{[1]_{\delta}}"'] 
		& W_c(\partial \Delta[n]) \arrow[d] \\
		{[1]}_{\square_c[n-1]}  \arrow[r]
		& W_n\ .
	\end{tikzcd}
 \]
\end{proof}

\begin{corollary}\label{corquillenadj}
The adjunction $W_c\dashv N^c$ is a Quillen adjunction for the Joyal model structure.
\end{corollary}

\begin{proof}
It follows from the fact that $W_c$ preserves cofibrations (Lemma \ref{lemmawpreservecof})
and weak equivalences (Lemma \ref{lemmareflects}).
\end{proof}

\begin{corollary}\label{lemmalurie}
 The adjunction $W_c \dashv N^c$ is a Quillen equivalence for the Joyal model structure.
\end{corollary}

\begin{proof}
 It is a consequence of the fact that the adjunction $L^{\Delta[1]}_c \circ W_c \dashv N^c \circ R^{\Delta[1]}_c$ relating simplicial sets to simplicial categories is a Quillen equivalence (see for instance \cite[Proposition 2.2.4.1]{Lurie09}) and the adjunction $L^{\Delta[1]}_c \dashv  R^{\Delta[1]}_c$ is a Quillen equivalence.
\end{proof}

\subsection{Nerve functors}

\begin{definition}
 Let $\catC$ be a cocomplete category. A functor $F: \Delta \to \catC$ induces an adjunction
 \[
\begin{tikzcd}
 \sSet \arrow[r, shift left, "F_!"] &  \catC\ , \arrow[l, shift left, "F^!"]
 \end{tikzcd}
\]
where the left adjoint $F_!$ is the left Kan extension of $F$ and where
\[
F^!(X)_n = \hom_{\catC} (F(n), X)\ .
\]
This right adjoint functor is called a nerve functor. If $\catC$ is a model category, the functor $F$ (or equivalently the functor $F^!$) is said to be homotopy coherent if the adjunction $F_! \dashv F^!$ is a Quillen adjunction with respect to the Joyal model structure.
\end{definition}

We are interested by the case where $\catC$ is the category $\catoE$ equipped with the Dwyer-Kan model structure; where $\catE$ is a monoidal model category. We know that the adjunction $i\dashv S$ relating $\catE$ to sets is monoidal. Thus, it extends to an adjunction also denoted $i \dashv S$ which relates $\catE$-categories to small categories.
\[
	\begin{tikzcd}[ampersand replacement=\&]
		\cato \arrow[r,shift left,"i"] \& \catoE \arrow[l,shift left,"S"]
	\end{tikzcd}
\]
Since the category $\Delta$ is a full subcategory of the category $\cato$, this provides us with a cosimplicial object in $\catoE$, that is 
\[
i([-]) : n \mapsto i([n])\ ,
\]
that we refer to using the notation $n \mapsto [n]$.

\begin{theorem}\label{thm:nerve}
 Let $\catE$ be a monoidal model category and suppose that the category $\catoE$ has a Dwyer-Kan model structure. Then, for any Reedy cofibrant replacement $F$ of the cosimplicial $\catE$-enriched category $n \mapsto [n]$, the nerve $F^!$ is homotopy coherent (that is the adjunction $F_!\dashv F^!$ is a Quillen adjunction).
\end{theorem}

Note first that the fact that $F$ is Reedy cofibrant implies that the functor $F_!: \sSet \to \catoE$ preserves cofibrations. So, it suffices to check that it preserves weak equivalences. One way to prove it is to show that the functor $F_!$ sends the maps
\[
	\Lambda^k[n] \to \Delta[n]\  ,\ 0 < k < n\ ,
\]
and the map $* \to N (* \leftrightarrow *)$ to equivalences. We actually take a shortcut by using the fact that $W_c$ is a left Quillen functor.\\

Recall that the category $\mathrm{Fun}(\Delta, \catoE)$ of cosimplicial $\catE$-categories carries a Reedy model structure where a map $F \to G$ is a cofibration (resp. an acyclic cofibration) if the map
\[
	F_!(\Delta[n]) \sqcup_{F_!(\partial \Delta[n])} G_!(\partial \Delta[n]) \to G_!(\Delta[n])\ ,
\]
is a cofibration (resp. an acyclic cofibration) for any integer $n$.

\begin{lemma}\label{lemma:simpreedy}
 For any monomorphism of simplicial sets $X \to Y$ and any Reedy cofibration $F \to G$ in the category $\mathrm{Fun}(\Delta, \catoE)$, then the map
 \[
 	F_!(Y) \sqcup_{F_!(X)} G_!(X) \to G_!(Y)\ ,
 \]
 is a cofibration. Moreover, this is an acyclic cofibration if $F \to G$ is an acyclic cofibration.
\end{lemma}

In particular, for any simplicial set $X$ and for any Reedy acyclic cofibration $F \to G$, the map $F_!(X) \to G_!(X)$ is an acyclic cofibration.

\begin{proof}
The proof is the same as Proposition \ref{prop:reedy2} using the standard Reedy structure on the category $\Delta$.
\end{proof}

\begin{lemma}\label{lemmaonegivesall}
 If a Reedy cofibrant replacement $F$ of the cosimplicial $\catE$-category $n \mapsto [n]$ is homotopy coherent, then all its Reedy cofibrant replacements are homotopy coherent.
\end{lemma}

\begin{proof}
Let us suppose that $F$ is homotopy coherent.
 Let $G$ be a cofibrant replacement of $[-]$. Then $G_!$ preserves cofibrations by Lemma \ref{lemma:simpreedy}. Thus it suffices to show that $G_!$ preserves weak equivalences. Let us consider the following factorisation in the Reedy model category of cosimplicial $\catE$-categories
\[
	\begin{tikzcd}
		F \sqcup G \arrow[r,hook] & G' \arrow[r,"\sim"] & {[-]}\ .
	\end{tikzcd}
 \]
Since $F$ and $G$ are cofibrant, then, the morphisms $F \to G'$ and $G \to G'$ are both acyclic cofibrations for the Reedy model structure. Thus, by Lemma \ref{lemma:simpreedy} for any simplicial set $X$, the two maps $F_!(X) \to G'_!(X) \leftarrow G_!(X)$ are weak equivalences. Subsequently, for any Joyal weak equivalence of simplicial sets $f: X \to Y$, $F_!(f)$ is a weak equivalence if and only if
$G'_!(f)$ is a weak equivalence if and only if 
$G_!(f)$ is a weak equivalence. To conclude, since $F_!$ preserves weak equivalences, then
$G_!$ preserves weak equivalences.
\end{proof}

\begin{lemma}\label{lemmareflects2}
A morphism $F: \CCC \to \DDD$ in $\catocrSet$ is a Dwyer-Kan equivalence if
and only if  $L_c^{\square_c[1]}(F)$ is a Dwyer-Kan equivalence.
\end{lemma}

\begin{proof}
This follows from the same arguments as those used in Lemma \ref{lemmareflects}, using
the fact the the functor $L_c^{\square_c[1]}: \crSet \to \ccSet$ reflects weak
equivalences.
\end{proof}

\begin{proof}[Proof of Theorem \ref{thm:nerve}]
First, let $G$ be a Reedy cofibrant replacement of the functor
\begin{align*}
 \Delta &\to \catocrSet\\
 n &\to [n]\ .
\end{align*}
Then the functor $L^{\square_c[1]}_p \circ G$ is also Reedy cofibrant. Besides, by Corollary \ref{corquillenadj}, the functor
\begin{align*}
 \Delta &\to \catoccSet\\
 n &\to W_n\ .
\end{align*}
is an homotopy coherent Reedy cofibrant replacement of the functor $n \to [n]$. So by Lemma \ref{lemmaonegivesall}, $L^{\square_c[1]}_p \circ G$ is homotopy coherent. Since the functor $L^{\square_c[1]}_p$ reflects weak equivalences (Lemma \ref{lemmareflects2}), then $G$ is also homotopy coherent. Then, for any interval $H$ of $\catE$, the cosimplicial $\catE$-enriched category $L^H_p \circ G$ is homotopy coherent. So again by Lemma \ref{lemmaonegivesall}, $F$ is homotopy coherent.
\end{proof}

\begin{proposition}\label{propnerveeq}
We use the same notation as in Theorem \ref{thm:nerve}. The following propositions are equivalent. 
\begin{enumerate}
 \item There exists an interval $H$ of $\catE$ such that the adjunction $L^H_p\dashv R^H_p$ relating precubical sets to $\catE$ is a Quillen equivalence.
 \item For any Reedy cofibrant replacement $F$ of the cosimplicial object $n \mapsto [n]$ of $\catoE$ the adjunction $F_! \dashv F^!$ is a Quillen equivalence.
 \item There exists a Reedy cofibrant replacement $F$ of the cosimplicial object $n \mapsto [n]$ of $\catoE$ such that the adjunction $F_! \dashv F^!$ is a Quillen equivalence.
\end{enumerate}
\end{proposition}

\begin{proof}[Proof of Proposition \ref{propnerveeq}]
The equivalence between $(2)$ and $(3)$ follows from the same arguments as those used to prove corollary \ref{corquilleneq}. Then, let $F : \Delta \to \catocrSet$ be a Reedy cofibrant replacement of $n \to [n]$. By the equivalence between $(2)$ and $(3)$ (for $\catE = \ccSet$) and by Corollary \ref{lemmalurie}, the adjunction $L^{\square_c[1]}_p \circ F_! \dashv F^! \circ R^{\square_c[1]}_p$ is a Quillen equivalence. Besides, the adjunction
\[
\begin{tikzcd}
\catocrSet \arrow[r, shift left, "L^{\square_c[1]}_p"] &  \catoccSet \arrow[l, shift left, "R^{\square_c[1]}_p"]
\end{tikzcd}
\]
is a Quillen equivalence. So, by the 2-out-of-3 rule, the adjunction $F_! \dashv F^!$ is also a Quillen equivalence.
Now, suppose $(1)$. Then, the adjunction $L^H_p \circ F_! \dashv F^!\circ R^H_p$ is a Quillen equivalence, which implies $(2)$.
Conversely, suppose $(3)$, then for any interval $H$ of $\catE$, the adjunction $L^H_p \circ F_! \dashv F^!\circ R^H_p$ is a Quillen equivalence
and since $F_! \dashv F^!$ is a Quillen equivalence, then the adjunction
\[
\begin{tikzcd}
\catocrSet \arrow[r, shift left, "L^{H}_p"] &  \catoE \arrow[l, shift left, "R^{H}_p"]
\end{tikzcd}
\]
is also a Quillen equivalence. In particular, for any precubical set $X$ and any fibrant object $Y$ of $\catE$, a morphism
$L^H_p([1]_X) \to [1]_Y$ is an equivalence if and only if the adjoint morphism $[1]_X \to R^H_p [1]_Y$ is an equivalence. This rewrites as:
$L^H_p(X) \to Y$ is an equivalence if and only if the adjoint morphism $X \to R^H_p Y$ is an equivalence ; that is, the adjunction 
\[
\begin{tikzcd}
\crSet \arrow[r, shift left, "L^{H}_p"] &  \catE \arrow[l, shift left, "R^{H}_p"]
\end{tikzcd}	
\]
is a Quillen equivalence.
\end{proof}


\section{Applications}

The goal of this final section is to describe various contexts where cubical categories appear.\\

Let $(\catE,\otimes ,\II)$ be a monoidal model category and let $H$ be a monoidal interval. We know that  it induces a Quillen monoidal adjunction $L_c^H\dashv R_c^H$ relating cubical sets with connections to $\catE$ which extends to the level of enriched categories.
\[
\begin{tikzcd}
 \catoccSet \arrow[r, shift left, "L^H_c"] &  \mathsf{Cat}_\catE\ . \arrow[l, shift left, "R^H_c"]
 \end{tikzcd}
\]
Moreover, any $\catE$-model category $\catM$ has an induced structure of a $\ccSet$-model category. In this section, we describe three examples of such a monoidal model category $\catE$: the simplicial sets with the Joyal model structure, the chain complexes and the differential graded coalgebras.

\subsection{A remark about the Boardman--Vogt construction}

A theory of homotopy coherent nerve is developed in \cite[\S 6]{MoerdijkWeiss07}. Roughly, for any monoidal model category $\catE$ (satisfying some conditions) equipped with a monoidal interval $H$, there exists a endofunctor $W_H: \catoE \to \catoE$ called the Boardman--Vogt construction together with a natural transformation $W_H \to \mathrm{Id}$ such that the functor
 \[
 	W_H \CC \to \CC
 \]
 is a cofibrant replacement of $\CC$ provided the unit maps $\II \to \CC(x,x)$ are cofibrations. More generally, any functor $F : \CC \to \DD$ which is injective on objects and such that the maps $F_{x,y} :\CC(x,y) \to \DD(F(x), F(y))$ are cofibrations in $\catE$ induces a cofibration of $\catE$-enriched categories
 \[
 	W_H(F): W_H \CC \to W_H \DD\ .
 \]
 Then, the functor $n \in \Delta \mapsto W_H [n] \in \catoE$ induces an adjunction
\[
\begin{tikzcd}
\sSet \arrow[r, shift left, "W_H"] &  \catoE\ , \arrow[l, shift left, "N^H"]
\end{tikzcd}
\]
where $N^H$ is an homotopy coherent nerve.

\begin{lemma}
The functor $W_H$ is isomorphic to the composite functor $L^H_c \circ W_c$.
\end{lemma}

\begin{proof}
 It follows from the fact that the functor $n \mapsto W_H[n]$ is isomorphic to the functor $n \mapsto L^H_c  W_n$.
\end{proof}

Therefore, the functor $N^H$ is isomorphic to the functor $N^c \circ R^H_c$.

\subsection{The underlying $(\infty,1)$-category of an $(\infty,2)$-category}

Let us endow the category of simplicial sets with the Joyal model structure which is a monoidal model structure. 

\begin{definition}
 Let $\III$ be the groupoid with two objects $0$ and $1$ such that 
 \[
 \III(i,j) = *\ ,\  \forall i,j \in \{ 0,1 \}\ . 
 \]
\end{definition}

In this context, the simplicial nerve of the groupoid $\III$
\[
H= N(\III)
\]
is a monoidal interval. Subsequently, there exists a monoidal Quillen adjunction relating cubical sets to simplicial sets with the Joyal model structure.
\[
\begin{tikzcd}
 \ccSet \arrow[r, shift left, "L^H_c"] &  \sSet_{J} \arrow[l, shift left, "R^H_c"] \end{tikzcd}
\]

\begin{proposition}
 Let $X$ be a quasi-category. Then $R^H_c(X)$ is canonically equivalent to $R^{\Delta[1]}_c(\mathrm{Core}(X))$ where $\mathrm{Core}(X)$ is the maximal Kan complex contained in $X$.
\end{proposition}

\begin{proof}
On the one hand, for any integer $n$, we have 
\begin{align*}
 \hom_{\crSet} (\square_c[n], R^H_c (X)) &\simeq \hom_{\sSet}(L^H_c (\square_c[n]),X)\\
 &\simeq \hom_{\sSet} (H^{\otimes n}, X)\\
 &\simeq \hom_{\sSet} (N(\III^n), X)\\
 &\simeq \hom_{\mathrm{quasi-categories}}(N(\III^n),X)\\
 &\simeq \hom_{\mathrm{Kan-complexes}}(N(\III^n),\mathrm{Core}(X))\text{ since }N(\III^n)\text{ is a Kan complex}\\
 &\simeq \hom_{\mathrm{\sSet}}(N(\III^n),\mathrm{Core}(X))\\
 &\simeq \hom_{\crSet} (\square_c[n], R^H_c (\mathrm{Core}(X))).
\end{align*}
Therefore, the canonical map $R^H_c (\mathrm{Core}(X)) \to R^H_c (X)$ is an isomorphism. On the other hand, in the Kan--Quillen model structure on simplicial sets whose fibrant objects are Kan complexes, the inclusion $\Delta[1] \to N(\III)$ is an equivalence. So, by Proposition \ref{prop:changeint}, we obtain an equivalence $R^{\Delta[1]}_c(\mathrm{Core}(X)) \to R^{H}_c(\mathrm{Core}(X))$.
\end{proof}

The adjunction $L^H_c\dashv R^H_c$ extends to the level of enriched categories.
\[
\begin{tikzcd}
 \catoccSet \arrow[r, shift left, "L^H_c"] &  \catosSet \arrow[l, shift left, "R^H_c"] \end{tikzcd}
\]
where the category $\catosSet$ of simplicial categories is equipped with the Dwyer-Kan model structure induced by the Joyal model structure on simplicial sets. Any simplicial category $\mathcal C$ whose mapping objects are quasi-categories (which is the case for any fibrant object of $\catosSet$), represents an $(\infty,2)$-category. The underlying $(\infty,1)$-category is the simplicial category with the same objects but whose mapping space between any two objects $x$ to $y$ is
\[
	\mathrm{Core}(\CC(x,y))\ .
\]
Let us denote it by $\mathrm{Core}(\CC)$. Then, the $(\infty,1)$-category represented by $\mathrm{Core}(\CC)$ is equivalent to the $(\infty,1)$-category represented by $R^{\Delta[1]}_c(\mathrm{Core}(\CC))$ (indeed, they yield the same quasi-category from the usual homotopy coherent nerve  functors). The proposition just above implies that 
\[
	R^{\Delta[1]}_c(\mathrm{Core}(\CC)) \simeq R^H_c \CC\ .
\]
Hence, $R^H_c \CC$ represents the underlying $(\infty,1)$-category of $\CC$.\\

Besides, by Proposition \ref{propcubicalmodelcat}, any $\sSet_J$-model category inherits a structure of a cubical model category. For instance, we have the following proposition.

\begin{proposition}
 The Joyal model category $\sSet_J$ is a $\ccSet$-model category.
\end{proposition}

\subsection{The dg nerve}

\begin{remark}\label{newremark2}
 Note that this example already appeared in \cite{RiveraZenalian16}.
\end{remark}

Let $\mbk$ be a commutative ring. We denote by $\dgMod$ the category of chain complexes of $\mbk$-modules. When equipped with the projective model structure, this is a monoidal model category. The following chain complex
\[
\begin{cases}
 C[1]_{0} = \mbk \cdot (0) \oplus  \mbk \cdot (1)\\
 C[1]_1 =  \mbk \cdot (01)\\
 C[1]_k = 0\text{ if }k\notin \{0,1\}\\
 d(01)= (1)-(0)\ .
\end{cases}
\]
has the structure of a monoidal interval when equipped with the maps $\delta^0, \delta^1: \mbk \to C[1]$, $\sigma: C[1]\to \mbk$ and $\gamma: C[1] \otimes C[1] \to C[1]$ defined by the formulas
\[
\begin{cases}
\delta^0 (1) = (0)\\
\delta^1 (1) = (1)\\
\sigma(0) = \sigma (1) =1\\
\sigma (01)=0\\
 \gamma((0)\otimes x) = \gamma( x\otimes (0)) =x\\
 \gamma((1)\otimes (1)) = (1)\\
 \gamma((01)\otimes (1)) = \gamma((1)\otimes (01)) =0\ .
\end{cases}
\]

This gives an adjunction $W_{C[1]} \dashv N^{C[1]}$ relating differential graded categories to simplicial set. The right adjoint functor $N^{C[1]}$ is the dg-nerve of dg-categories described in \cite[\S 1.3.1]{Lurie12}. Besides, any $\dgMod$-model category has an induced structure of a $\ccSet$-model category.

\begin{proposition}
 The category of chain complexes of $\mbk$-modules is a $\ccSet$-model category.
\end{proposition}

\subsection{The coalgebraic nerve}

Here $\mbk$ is a field.

\subsubsection{A coalgebraic model of the interval}

\begin{definition}
 A counital coassociative coalgebra $(V, w, \tau)$ is a comonoid in the category of chain complexes. We denote by $\uCog$ the category of such coalgebras.
\end{definition}

Since $\mbk$ is a field, the category of counital coassociative coalgebras admits a monoidal model structure whose cofibrations and weak equivalences are respectively degreewise injections and quasi-isomorphisms ; see \cite{GetzlerGoerss99}.
The chain complex monoidal interval $C[1]$ has the structure of a coalgebra as follows
\[
\begin{cases}
 \tau (0) = \tau (1) =1\\
 w(i) = (i) \otimes (i) \text{ for }i\in \{0,1\}\\
 w(01)=(0)\otimes(01)+(01)\otimes (1)\ .
\end{cases}
\]
Moreover, a straightforward checking leads to well known following result.

\begin{lemma}
The data of $(C[1], w,\tau, \delta^0,\delta^1,\sigma, \gamma)$ defines a monoidal interval in the category of counital coassociative coalgebras. 
\end{lemma}

Therefore, any $\uCog$-model category has an induced structure of a $\ccSet$-model category. We will study the example of dg associative algebras.

\subsubsection{The $\uCog$-model category of $\Aa_\infty$-algebras.}

The remaining of this article is devoted to the description through our cubical approach of the higher structures appearing in the study in of associative algebras  in chain complexes over a field. What is done there can easily be extended to the case of algebras over a nonsymmetric operad using the theory developed in \cite{LeGrignou16a}.\\

\begin{definition}
 A differential graded (dg) associative algebra (or dg algebra for short) $(\Aa,m)$ is the data of a chain complex $\Aa$ together with an associative product $m :\Aa \otimes\Aa \to \Aa$.
\end{definition}

\begin{definition}
 A dg conilpotent coassociative coalgebra (or dg conilpotent coalgebra for short) $(\CC, w)$ is the data of a chain complex $\CC$ together with a coproduct $w :\CC \to \CC \otimes\CC $ which is coassociative that is $(w \otimes \mathrm{Id}) \circ w = (\mathrm{Id} \otimes w) \circ w$ and conilpotent, that is, for any element $x \in \CC_m$, there exists an integer $n$ such that
 \[
 	w^{(n)}(x) := (w \otimes \mathrm{Id}^{\otimes n-1}) \circ \cdots \circ  (w \otimes \mathrm{Id}) \circ w (x) =0\ .
 \]
\end{definition}

There exists an adjunction 
\[
\begin{tikzcd}
\NilCog \arrow[r, shift left, "\Omega"] &  \Alg \arrow[l, shift left, "B"] 
\end{tikzcd}
\]
relating dg algebras to dg conilpotent coalgebras. The right adjoint $B$ called the bar functor is defined as follows.
\begin{itemize}
 \item[$\triangleright$] The underlying graded coalgebra of $B \Aa$ is the cofree conilpotent coalgebra on the suspension of $\Aa$,
 \[
 	\ov T s\Aa := \bigoplus_{n \geq 1} s\Aa^{\otimes n}
 \]
 whose coproduct is given by
 \[
 	w(sx_1 \otimes\cdots \otimes  sx_n) = \sum_{i=1}^{n-1} (sx_1 \otimes \cdots sx_i) \otimes (sx_{i+1} \otimes\cdots \otimes  sx_n)\ .
 \]
 \item[$\triangleright$] The differential $d_{B\Aa}$ on $B\Aa$ is the only coderivation on $\ov T s\Aa$ whose projection on the cogenerators $s\Aa$ is the following composite map
\begin{align*}
 \overline{T} s\Aa \twoheadrightarrow s\Aa \oplus s\Aa \otimes s\Aa & \to s \Aa \\
 sx &\mapsto -sd_\Aa x\\
 sx \otimes sy &\mapsto (-1)^{|x|} s \gamma_{\Aa} (x \otimes y)
 \end{align*}

The fact that $d_{B\Aa}^2 = 0$ follows from the fact that $d_\Aa^2=0$, that $\gamma_\Aa$ is associative and that $\gamma_\Aa$ and $d_\Aa$ satisfy the Leibniz equation.
\end{itemize}

\begin{proposition}
 There exists a model structure on the category of dg algebras whose fibrations are degreewise surjections and whose weak equivalences are quasi-isomorphisms.
\end{proposition}

\begin{theorem}\cite{LefevreHasegawa03}
 There exists a model structure on the category dg conilpotent coalgebras transferred through the adjunction $\Omega \dashv B$, that is
\begin{itemize}
 \item[$\triangleright$] a cofibration is a morphism $f$ such that $\Omega(f)$ is a cofibration of algebras,
 \item[$\triangleright$] a weak equivalence is a morphism $f$ such that $\Omega(f)$ is a quasi-isomorphism.
\end{itemize}
Moreover, for any dg algebra $\Aa$, the morphism
 \[
 	\Omega B \Aa \to \Aa
 \]
 is a cofibrant replacement of $\Aa$, which implies that the adjunction $\Omega \dashv B$ is a Quillen equivalence.
\end{theorem}

\begin{proposition}\cite{AnelJoyal13}\cite{LeGrignou16a}
 The category $\NilCog$ and the category $\Alg$ are both $\uCog$-model categories. Moreover, there exists a natural isomorphism of counital coassociative coalgebras
 \[
 \{\CC,B\Aa \} \simeq \{\Omega \CC,\Aa \}\ ,
 \]
 for any dg conilpotent coalgebra $\CC$ and for any dg algebra $\Aa$.
\end{proposition}

Let us give a hint on what are these enrichments. On the one hand, dg algebras are canonically cotensored over counital coassociative coalgebras. Indeed, for any dg algebra $A$ and any counital coassociative coalgebra $V$, the chain complex $[V,A]$ has the canonical structure of a dg algebra called the convolution algebra. Then, the tensorisation and the enrichment may be obtained by the adjoint functor theorem. On the other hand, the category of dg conilpotent coalgebras is canonically tensored over counital coassociative coalgebras. Indeed, for any dg conilpotent coalgebra $\CC$ and any counital coassociative coalgebra $V$, the chain complex $\CC \otimes V$ has the canonical structure of a dg conilpotent coalgebra. Then, the cotensorisation and the enrichment may be obtained with the adjoint functor theorem.

\subsubsection{The infinity category of dg algebras}

Restricting the $\uCog$-enriched category of dg conilpotent coalgebras to bar constructions of dg algebras (which are in particular fibrant-cofibrant dg conilpotent coalgebras) we obtain an $\uCog$-enriched category which is essentially the same as the $\uCog$-enriched category $\Aa lg$ whose objects are dg algebras and such that
 \[
 	\Aa lg (\Aa, \Aa') := 	 \{B\Aa,B\Aa'\}\ .
 \]
Then, using the interval $C[1]$, one obtains a $\ccSet$-enriched category $\Aa lg^\square$ with the same objects as $\Aa lg$ and such that
\begin{align*}
 	\Aa lg^\square (\Aa, \Aa')(n) &= \hom_{\NilCog} (B\Aa \otimes C[1]^{\otimes n}, B\Aa')\\
	&\simeq \hom_{\uCog} (C[1]^{\otimes n} ,\{ B\Aa, B\Aa'\})\\
	&\simeq \hom_{\uCog} (C[1]^{\otimes n} ,\{ \Omega B\Aa, \Aa'\})\\
	&\simeq \hom_{\Alg} (\Omega B\Aa, [C[1]^{\otimes n},\Aa'])\ .
\end{align*}
Finally, the nerve $N^{c}\Aa lg^\square$ of this cubical category is the quasi-category whose $n$ vertices are the data of
\begin{itemize}
 \item[$\triangleright$] $n+1$ dg algebras $\Aa_0,\ldots,\Aa_n$,
 \item[$\triangleright$] for any integers $0 \leq i<j \leq n$, a morphism of dg conilpotent coalgebras
 \[
 	f_{i,j}: (B \Aa_i ) \otimes C[1]^{\otimes j-i-1} \to B \Aa_j\ ,
 \]
 which is equivalent to the data of a morphism of dg algebras 
 
\[
	\Omega B \Aa_i  \to [ C[1]^{\otimes j-i-1} , \Aa_j]\ ,
\]
 \item[$\triangleright$] such that the following diagram commutes 
 \[
	\begin{tikzcd}[ampersand replacement=\&]
		B \Aa_i \otimes C[1]^{\otimes j-i-1} \otimes C[1]^{\otimes k-j-1} \arrow[r,"f_{i,j} \otimes \mathrm{Id}"] \arrow[d]
		\& B \Aa_j \otimes C[1]^{\otimes k-j-1} \arrow[dd,"f_{j,k}"]\\
		B \Aa_i \otimes C[1]^{\otimes j-i-1} \otimes \II \otimes C[1]^{\otimes k-j-1} \arrow[d,"\mathrm{Id} \otimes \delta^1 \otimes \mathrm{Id}"]\\
		 B \Aa_i \otimes C[1]^{\otimes k-i-1} \arrow[r,"f_{i,k}"']
		 \& B \Aa_k
	\end{tikzcd}
\]
for any integers $0 \leq i < j < k \leq n$.
\end{itemize}

\bibliographystyle{amsalpha}

\begin{thebibliography}{MW07}

\bibitem[AJ13]{AnelJoyal13}
Mathieu Anel and Andr\'e Joyal, \emph{Sweedler theory of (co)algebras and the
  bar-cobar constructions}, arXiv:1309.6952 (2013).

\bibitem[B.19]{LeGrignou16a}
Le~Grignou B., \emph{Homotopy theory of unital algebra}, Algebraic And
  Geometric Topology \textbf{19} (2019), 1541–1618.

\bibitem[BM06]{BergerMoerdijk06}
Clemens Berger and Ieke Moerdijk, \emph{The {B}oardman-{V}ogt resolution of
  operads in monoidal model categories}, Topology \textbf{45} (2006), no.~5,
  807--849.

\bibitem[BM12]{BergerMoerdijk13}
\bysame, \emph{On the homotopy theory of enriched categories}, Quarterly
  Journal of Mathematics \textbf{64} (2012).

\bibitem[BP81]{BrownHiggins81}
R.~Brown and Higgins P., \emph{On the algebra of cubes}, Journal of pure and
  appiled algebra \textbf{21} (1981), no.~3, 233--260.

\bibitem[Cav14]{Caviglia14}
Giovanni Caviglia, \emph{{A Model Structure for Enriched Coloured Operads}},
  arXiv:1401.6983 (2014).

\bibitem[Cis06]{Cisinski06}
Denis-Charles Cisinski, \emph{Les pr\'efaisceaux comme mod\`eles des types
  d'homotopie}, Ast\'erisque, 2006.

\bibitem[Cis14]{Cisinski14}
\bysame, \emph{Univalent universes for elegant models of homotopy types}.

\bibitem[Day70]{Day70}
Brian Day, \emph{On closed categories of functors}, Reports of the Midwest
  Category Seminar IV, Lecture Notes in Mathematics \textbf{137} (1970), 1--38.

\bibitem[GG99]{GetzlerGoerss99}
Ezra Getzler and Paul Goerss, \emph{{A model category structure for
  differential graded coalgebras}},
  \texttt{http://www.math.northwestern.edu/$\sim$pgoerss/papers/model.ps}
  (1999).

\bibitem[GM03]{GrandisMauri03}
Marco Grandis and Luca~Johan Mauri, \emph{Cubical sets and their site}, Theory
  and Applications of Categories \textbf{11} (2003), no.~8, 185--211.

\bibitem[Hov99]{Hovey99}
Mark Hovey, \emph{Model categories}, Mathematical Surveys and Monographs,
  vol.~63, American Mathematical Society, Providence, RI, 1999.

\bibitem[Isa09]{Isaacson09}
S.B. Isaacson, \emph{Cubical homotopy theory and monoidal model categories},
  Ph.D. thesis, Harvard University, 2009. \MR{MR2713397}

\bibitem[Jar02]{Jardine02}
J.F. Jardine, \emph{Cubical homotopy theory: a beginning},
  http://www.math.uwo.ca/?jardine/papers/preprints/index.shtml. (2002).

\bibitem[Jar06]{Jardine06}
\bysame, \emph{Categorical homotopy theory}, Homology Homotopy Appl. \textbf{8}
  (2006), no.~1, 71--144.

\bibitem[Joy]{Joyal00}
A.~Joyal, \emph{The theory of quasi-categories and its applications},
  http://mat.uab.cat/~kock/crm/hocat/advanced--course/Quadern45--2.pdf.

\bibitem[Kel74]{Kelly74}
G.~M. Kelly, \emph{Doctrinal adjunction}, Category Seminar (Berlin, Heidelberg)
  (Gregory~M. Kelly, ed.), Springer Berlin Heidelberg, 1974, pp.~257--280.

\bibitem[KL01]{KellyLack00}
G.~M. Kelly and S.~Lack, \emph{V-cat is locally presentable or locally bounded
  if v is so}, TAC \textbf{8} (2001).

\bibitem[KV18]{KapulkinVoevodsky18}
Krzysztof Kapulkin and Vladimir Voevodsky, \emph{Cubical approach to
  straightening}.

\bibitem[LH03]{LefevreHasegawa03}
K.~Lefevre-Hasegawa, \emph{Sur les {A}-infini cat\'egories},
  \texttt{arXiv.org:math/0310337} (2003).

\bibitem[Lur09]{Lurie09}
Jacob Lurie, \emph{Higher topos theory}, Annals of Mathematics Studies, vol.
  170, Princeton University Press, Princeton, NJ, 2009.

\bibitem[Lur12]{Lurie12}
\bysame, \emph{Higher algebra},
  \texttt{www.math.harvard.edu/$\sim$lurie/papers/HigherAlgebra.pdf}, 2012.

\bibitem[Mal09]{Maltsiniotis09}
Georges Maltsiniotis, \emph{La cat\'egorie cubique avec connexions est une
  cat\'egorie test stricte}, Homology Homotopy Appl. \textbf{11} (2009), no.~2,
  309--326.

\bibitem[MW07]{MoerdijkWeiss07}
Ieke Moerdijk and Ittay Weiss, \emph{Dendroidal sets}, Algebr. Geom. Topol.
  \textbf{7} (2007), 1441--1470. \MR{2366165 (2009d:55014)}

\bibitem[Ros09]{Rosicky09}
Jiri Rosicky, \emph{On combinatorial model categories}, Appl. Cat. Str.
  \textbf{17} (2009), 303--316.

\bibitem[RZ18]{RiveraZenalian16}
Manuel Rivera and Mahlmoud Zeinalian, \emph{Cubical rigidification, the cobar
  construction and the based loop space}, Algebraic And Geometric Topology
  \textbf{18} (2018), 3789–3820.

\end{thebibliography}
\def\cprime{$'$}
\providecommand{\bysame}{\leavevmode\hbox to3em{\hrulefill}\thinspace}
\providecommand{\MR}{\relax\ifhmode\unskip\space\fi MR }
\providecommand{\MRhref}[2]{%
  \href{http://www.ams.org/mathscinet-getitem?mr=#1}{#2}
}
\providecommand{\href}[2]{#2}

\end{document}